\documentclass[a4paper]{amsart}

\usepackage{marginnote}
\usepackage{amsmath}
\usepackage{amsfonts}
\usepackage{amsthm}
\usepackage{amssymb}
\usepackage[all]{xy}
\usepackage[utf8]{inputenc}
\usepackage{mathrsfs}

\numberwithin{equation}{section}

\newtheorem{theorem}{Theorem}[section]
\newtheorem{lemma}[theorem]{Lemma}
\newtheorem{proposition}[theorem]{Proposition}
\newtheorem{corollary}[theorem]{Corollary}
\newtheorem{definition}[theorem]{Definition}
\newtheorem{remark}[theorem]{Remark}

\DeclareMathOperator{\Lie}{Lie}

\DeclareMathOperator{\ad}{ad}
\DeclareMathOperator{\Ad}{Ad}

\DeclareMathOperator{\supp}{supp}

\newcommand{\pr}[3]{\mathrm{pr}_{#2,#3}\left(#1\right)}
\newcommand{\res}[2]{{#1}_{|#2}}
\newcommand{\tg}[2]{\mathrm{T}_{#2}(#1)}
\newcommand{\cent}[2]{\mathfrak{z}_{#1}(#2)} 
\newcommand{\inv}[2]{#1^{\mathrm{#2}}}
\newcommand{\len}[1]{\lVert #1 \rVert}

\newcommand{\biform}[2]{\langle #1,#2\rangle}
\newcommand{\conj}[1]{#1^{\mathrm{conj}}}

\newcommand{\SL}{\mathrm{SL}}
\newcommand{\SO}{\mathrm{SO}}

\newcommand{\R}{\mathbb{R}}
\newcommand{\C}{\mathbb{C}}

\newcommand{\iu}{\mathrm{i}}

\newcommand{\rk}{\mathrm{rk}}

\newcommand{\sX}{\mathrm{X}}
\newcommand{\sY}{\mathrm{Y}}
\newcommand{\sZ}{\mathrm{Z}}
\newcommand{\sE}{\mathrm{E}}
\newcommand{\sS}{\mathrm{S}}

\newcommand{\id}{\mathrm{id}}

\newcommand{\bc}{\mathbf{c}}

\newcommand{\fg}{\mathfrak{g}}
\newcommand{\fk}{\mathfrak{k}}
\newcommand{\fs}{\mathfrak{s}}
\newcommand{\fq}{\mathfrak{q}}
\newcommand{\fa}{\mathfrak{a}}

\newcommand{\rG}{G}
\newcommand{\rK}{K}
\newcommand{\rA}{A}
\newcommand{\rN}{N}
\newcommand{\rM}{M}
\newcommand{\rW}{\mathfrak{w}}
\newcommand{\HC}{\mathrm{H}} 
\newcommand{\Rr}{\mathit{\Delta}}

\newcommand{\smo}{C^{\infty}}

\newcommand{\ct}{\tilde{t}}
\newcommand{\cx}{\tilde{x}}
\newcommand{\cka}{\tilde{\kappa}}

\newcommand{\crF}{\mathscr{F}}
\newcommand{\crFI}{\mathscr{F}_I}
\newcommand{\crH}{\mathscr{H}}
\newcommand{\crP}{\mathscr{P}}
\newcommand{\crS}{\mathscr{S}}

\newcommand{\invdiff}{\mathscr{L}_{\mathrm{Diff}}}

\newcommand{\DKV}{Duistermaat, Kolk and Varadarajan }

\newcommand{\quuad}{\quad\quad\quad\quad\quad\quad\quad\quad}

\begin{document}

\title{An estimate for spherical functions on $\SL(3,\R)$}
\author{Xiaocheng Li}
\address{Beijing International Center for Mathematical Research, Peking University, No. 5 Yiheyuan Road, Haidian District, Beijing, 100871, China P. R.}
\email{lixiaocheng@bicmr.pku.edu.cn}
\date{\today}

\maketitle

\begin{abstract}
We prove an estimate for spherical functions $\varphi_\lambda(a)$ on $\mathrm{SL}(3,\mathbb{R})$, establishing uniform decay in the spectral parameter $\lambda$ when the group variable $a$ is restricted to a compact subset of the abelian subgroup $\rA$. In the case of $\mathrm{SL}(3,\mathbb{R})$, it improves a result by J.J. Duistermaat, J.A.C. Kolk and V.S. Varadarajan by removing the limitation that $a$ should remain regular. As in their work, we estimate the oscillatory integral that appears in the integral formula for spherical functions by the method of stationary phase. However, the major difference is that we investigate the stability of the singularities arising from the linearized phase function by classifying their local normal forms when $\lambda$ and $a$ vary.


\end{abstract}

\tableofcontents

\section{Introduction}
\label{intro}

It is well-known that spherical functions, also called elementary spherical functions play many roles in representation theory and harmonic analysis on semisimple Lie groups, see Gangolli and Varadarajan \cite{GV} and Helgason \cite{Helg}. In this paper, we prove an asymptotic estimate for spherical functions on the simple Lie group $\SL(3,\R)$ which has real rank 2. 

\subsection{Notations}

In order to state the results, we are going to introduce the common setting of semisimple Lie groups for $\rG=\SL(3,\R)$. Fix an Iwasawa decomposition $\rG=\rK\rA \rN$, here $\rK$ is $\SO(3)$, $\rA$ is the subgroup of all diagonal matrices with positive diagonal entries, $\rN$ is the nilpotent subgroup of all upper triangular matrices with diagonal entries 1. Let $\fg,\fk,\fa$ be the Lie algebras of $\rG,\rK,\rA$, respectively.  The Iwasawa projection $\HC:\rG\to \fa$ is defined by that for $g\in \rG$, \[g=k\cdot \exp \HC(g)\cdot n,\quad  k\in \rK,n\in \rN.\]
Let $B$ denote the Killing form on $\fg$. Let $\theta$ be the Cartan involution of $\fg$, $\theta(X)=-X^T$ for $X\in\fg$. The Cartan decomposition associated to $\theta$ is $\fg=\fk+\fs$. A norm on $\fg$ is defined by $\len{X}_{\fg}^2=-B(X,\theta X)$ for $X\in \fg$. Let $\fq=\fs\ominus \fa$ be the orthogonal complement of $\fa$ in $\fs$ with respect to $B$. Let $\rM$, resp. $\rM'$ be the centralizer, resp. normalizer of $\fa$ in $\rK$. Let $\rW$ denote the Weyl group $\rM'/\rM$. Let $\Rr$ be the set of roots of $(\fg,\fa)$. For $\alpha\in \Rr$, $\fg_\alpha$ is the corresponding root space. According to the Iwasawa decomposition, we choose the positive system of roots $\Rr^+\subseteq \Rr$ and the positive Weyl chamber $\fa^+\subseteq \fa$.

Let $\crF$ be the complex dual of the complexification $\fa_\C$ of $\fa$. $\rW$ acts naturally on $\crF$. We denote by $\crFI$, resp. $\crF_R$ the $\R$-linear subspace of $\crF$ of all elements that take only purely imaginary, resp. real values on $\fa$. $\crFI$, resp. $\crF_R$ could be identified with $\iu \fa^*$, resp. $\fa^*$. For $\lambda\in \crF$, it could be written as $\lambda_R+\lambda_I$ with $\lambda_R\in \crF_R, \lambda_I\in\crFI$. Then the conjugation $\conj{\lambda}$ of $\lambda$ is $\lambda_R-\lambda_I$. Take the restriction of $B$ to $\fa\times \fa$, then extend it to a complex bilinear form $\biform{\cdot}{\cdot}$ on $\fa_\C\times \fa_\C$. We identify $\crF$ with $\fa_\C$ under $\biform{\cdot}{\cdot}$. For $\lambda\in \crF$, denote by $\lambda^\vee\in\fa_\C$ the corresponding element. Then we transfer  $\biform{\cdot}{\cdot}$ to $\crF\times \crF$. A norm $\len{\cdot}_{\crF}$ on $\crF$ is defined by $\len{\lambda}_{\crF}^2=\biform{\lambda}{\conj{\lambda}}$. For the sake of efficiency, we omit the subscripts in $\len{\cdot}_\fg$ and $\len{\cdot}_\crF$.

The integral formula for spherical functions proved by Harish-Chandra (see \cite[Proposition 3.1.4]{GV}) reads that for $\lambda\in \crF$
	\begin{equation}\label{eq:insp}
		\varphi_\lambda(g)=\int_{\rK}  e^{(\lambda-\rho)\HC(gk)}dk,\quad g\in \rG.
	\end{equation}
Here $\rho=\frac{1}{2}\sum_{\alpha\in \Rr^+} \dim(\fg_\alpha) \alpha\in \fa^*$, while $\dim(\fg_\alpha)=1$ for $\alpha\in \Rr^+$. For the spherical function $\varphi_\lambda(g)$, we call $g$ the group variable and $\lambda$ the spectral parameter. It is well-known that $\varphi_\lambda$ is $\rK$-bi-invariant. By $\rG=\rK\rA\rK$, we always consider $\varphi_\lambda(a)$ for $a\in\rA$.

\subsection{Main result} The asymptotic behavior of spherical functions when the group variable $g$ goes to infinity has been carefully studied, starting with the classical work of Harish-Chandra. However, some analysis problems require the understanding of the asymptotic behavior when the spectral parameter $\lambda$ goes to infinity. Here $\lambda$ goes infinity means that the imaginary part $\lambda_I$ goes to infinity while the real part $\lambda_R$ remains bounded. For our purpose, we could assume $\lambda_R=0$, that is, $\lambda\in \crFI$. One approach for this problem was carried out by \DKV \cite{DKV83}. They employed the method of stationary phase and achieved the estimate for spherical functions on general semisimple Lie groups. Their estimate is uniform in the spectral parameter $\lambda$ but becomes not sharp once allowing the group variable to be singular. Nevertheless, this paper is inspired by their work. 


The major aim of this paper is to prove the following asymptotic estimate for spherical functions on the specific group $\rG=\SL(3,\R)$ which improves the result implied by \cite{DKV83}.

\begin{theorem} \label{est-sphfun}
For any compact subset $\omega$ of $\fa$, we have
\begin{equation}\label{eq:bdsp}
	|\varphi_\lambda(\exp H)|\lesssim_\omega \sum_{s\in \rW} \Omega(s^{-1}H,-\iu\lambda^\vee)^{-1/2},\quad \forall H\in\omega,\lambda\in \crFI.
\end{equation}
Here $\Omega:\fa\times \fa\to  [1,\infty)$ is defined by
		\[\Omega(H,H'):=\prod_{\alpha\in \Rr^+} (1+|\alpha(H)\alpha(H')|).\]
\end{theorem}


The above estimate only deals with $\lambda\in \crFI$. Actually, Proposition \ref{est-2} implies an estimate when the real part $\lambda_R$ of $\lambda$ is bounded. In this paper, we adopt the notation $\lesssim$. The implicit constants could depend on some terms. We would only specify the terms which are of our interest.

To illustrate the difference between our result and \cite{DKV83}, we present the estimate implied by \cite[Theorem 11.1]{DKV83} here. Let $\omega$ be a compact subset of $\fa$. Put 
	$\Rr^+(\omega)=\{\alpha\in \Rr^+|\alpha(H)\neq 0, \forall H\in \omega\}$.
We have
	\begin{equation}\label{eq:wbdsp}
		|\varphi_\lambda(\exp H)|\lesssim_\omega \sum_{s\in \rW}\prod_{\alpha\in \Rr^+(s^{-1}\omega)}(1+|\alpha(-\iu\lambda^\vee)|)^{-1/2},
	\quad \forall H\in \omega, \lambda\in \crFI.
	\end{equation}

Comparing the estimates \eqref{eq:bdsp} and \eqref{eq:wbdsp}, our estimate does not involve $\Rr^+(s^{-1}\omega)$ and $\alpha(s^{-1}H)$ is inserted in front of $\alpha(-\iu\lambda^\vee)$. When $\omega$ is inside a Weyl chamber, then $|\alpha(s^{-1}H)|$ stays away from 0. So the two estimates are comparable after ignoring the implicit constants. On the contrary, when $0\in \omega$, $\Rr^+(s^{-1}\omega)=\emptyset$ and estimate \eqref{eq:wbdsp} only tells that the spherical function is bounded while estimate \eqref{eq:bdsp} still gives uniform decay.

Now we turn to an application of our estimate, that is given in Section 6. Here we summarize the background and the result. 

On general  compact Riemannian manifolds, the possible concentrations of Laplace-Beltrami eigenfunctions have been widely studied. One way to measure concentrations is to study $L^p$ norms  $(2\leq p\leq \infty)$. Sogge \cite{Sog} studied the growth of $L^p$ norms of eigenfunctions on the whole manifold. Burq, Gérard, and Tzvetkov \cite{BGT} studied the growth of  $L^p$ norms of the restrictions of eigenfunctions to submanifolds. For some cases the general estimates are sharp but for some other cases they are far from optimal. A lot of research have been done. 

A natural idea is to investigate similar problems under specific settings. Sarnak discussed the case of locally symmetric spaces in his letter to Morawetz \cite{Sar}. It was suggested \cite[Page 14]{Sar} that we coud gain a power saving for the growth of the $L^\infty$ norm of  joint eigenfunctions on locally symmetric spaces of higher rank. Under this direction, Marshall \cite{M} proved upper bounds for the $L^p$ norms $(2\leq p\leq \infty)$ of joint eigenfunctions, as well as their restrictions to maximal flat subspaces. We remark that  a crucial ingredient in the investigation \cite{M} is the asymptotic estimate for spherical functions.

Benefited from our Theorem \ref{est-sphfun}, we prove the following result which strengthens Marshall \cite[Theorem 1.2]{M} for the case $\sS=\SL(3,\R)/\SO(3)$ and $p=2$.

\begin{theorem}\label{est-eigen0}
Let $\sX$ be a compact locally symmetric space that is a quotient of  $\sS=\SL(3,\R)/\SO(3)$. Let $\sE$ be an open ball in a maximal flat subspace in $\sX$ that is sufficiently small. If $\psi\in \smo(\sX)$ is a joint eigenfunction with spectral parameter $\lambda\in \crF$ and $\len{\psi}_{L^2(\sX)}=1$, then we have
	\[\len{\psi}_{L^2(\sE)}\lesssim_{\sE} \Omega_0(-\iu \lambda_I^\vee)^{1/4}.\]
The implicit constant depends on $\sE$.  Here $\Omega_0:\fa\to [1,\infty)$ is defined by
	\[\Omega_0(H):=\prod_{\alpha\in \Rr^+}(1+|\alpha(H)|).\]
\end{theorem}
 
\subsection{Strategy} 

Let us discuss our approach to Theorem \ref{est-sphfun}. As in \cite{DKV83}, we study the following oscillatory integral
	\begin{equation}\label{eq:ost1}
		\int_\rK u(k) e^{\lambda(\HC(\exp H \cdot k))}dk, \quad H\in\fa, \lambda\in \crFI, u\in \smo(\rK,\C). 
	\end{equation}
Adopting the conventions of harmonic analysis, we call $-\iu\lambda( \HC(\exp H \cdot k))$ the phase function ($\lambda$ is imaginary so we multiply $-\iu$) and $u(k)$ the amplitude of this oscillatory integral. It is clear that an estimate for general $u$ implies an estimate for the spherical functions. A general uniform decay estimate could be established because the method of stationary phase does not demand accurate information of $u$. We mention that for fixed $H$ and $\lambda$, it is not hard to derive an $O(t^{-\frac{1}{2}n(H,\lambda)})$ estimate for the upper bound at $t\lambda$ as $t\to \infty$. The difficulty is to make the estimate uniform in both $H$ and $\lambda$, that is, allowing them to vary. A helpful observation is that $n(H,\lambda)$ changes when  $H$ and $\lambda$ vary.

Unlike in \cite{DKV83}, we do not study the oscillatory integral \eqref{eq:ost1} directly because the nonlinearity of the Iwasawa projection makes the analysis complicated, which is reflected in \cite{DKV83}. Instead, with the aid of a result by Duistermaat \cite{D}, we study a related oscillatory integral
	\begin{equation}\label{eq:ost2}
		\int_\rK u(k) \exp {\lambda\left[\pr{\Ad (k^{-1})H}{\fs}{\fa}\right]}dk,  \quad H\in\fa, \lambda\in \crFI, u\in \smo(\rK,\C). 
	\end{equation}
Here $\mathrm{pr}_{\fs,\fa}:\fs\to \fa$ is the orthogonal projection with respect to the restriction of $B$ to $\fs\times \fs$. Let $H'\in \fa$ denote $-\iu\lambda^\vee$. We rewrite the phase function as
	\begin{equation}\label{eq:phase}
		-\iu\lambda\left[\pr{\Ad (k^{-1})H}{\fs}{\fa}\right]=B(\Ad (k^{-1})H,H')=B(H,\Ad (k) H').
	\end{equation}
Here an interesting feature is that $H$ and $H'$ play almost the same role. 

The phase function $B(H,\Ad(k) H')$ could be expressed as $tB(H_1,\Ad (k) H'_1)$ for $\len{H_1}=\len{H'_1}=1$ and $t>0$. The decay of the oscillatory integral \eqref{eq:ost2} mainly results from the increase in $t$.  Meanwhile, $H_1,H'_1$ also influence the decay since the critical set of $B(H_1,\Ad (k) H'_1)$ varies in an `unstable' manner as $H_1,H'_1$ switch from singular elements to regular elements. As the integral region $\rK$ has dimension 3, this instability leads to obstacles to the quantitative study of the oscillatory integral.

The novelty of the article is to resolve the instability by showing that the singularity of the phase function \eqref{eq:phase} is actually stable as the parameters vary.  The stability of the singularity refers to stable descriptions of the behavior of the phase function \eqref{eq:phase} near the critical points. More precisely, we uniformize the phase function into local normal forms, for instance, quadratic polynomials, see Remark \ref{rem-ori}. Section 4 is devoted to this task and a delicate technique is developed there. As a consequence, we manage to classify all the local normal forms concerning our estimate. 


After sketching our approach, curious readers may wonder why we focus on the particular case $\rG=\SL(3,\R)$ instead of general semisimple Lie groups. Let us explain the reasons here. First, it is not hard to modify Theorem \ref{est-sphfun} for general semisimple Lie groups. We expect that the general version holds and the methods in this paper would be generalized. Second, we point out that the estimate in question is trivial when the real rank $\dim \fa$ of the semisimple Lie group is 1. Considering the higher rank cases, the author believes that the case for $\SL(3,\R)$ could serve as a prime example and a detailed study would provide helpful insights into general higher rank cases.


\subsection{Relevant results} 

Let us discuss some relevant literature besides \cite{DKV83} about the decay estimate for spherical functions in the spectral parameter. Stanton and Tomas obtained a classic result \cite[Theorem 2.1]{ST} for expansions of spherical functions on semisimple groups of real rank one. Clerc studied the generalized Bessel functions in \cite{C} which relates the oscillatory integral \eqref{eq:ost2}. Cowling and Nevo obtained a relevant result \cite[Theorem 1.1]{CN} but their setting requires complex semisimple Lie groups. As mentioned before, the estimate implied by \cite[Theorem 11.1]{DKV83} is trivial when $0\in \omega$, here $\omega$ is a compact subset of $\fa$. This weakness leads to difficulty in applications. 

In recent years, due to applications in number theory, in order to alleviate this weakness, the result and method in \cite{DKV83} have been reconsidered and some progress has been made. According to our knowledge, we mention following results for readers' convenience. Marshall obtained an estimate \cite[Theorem 1.3]{M} which is uniform in the group variable while the spectral parameter is required to remain regular. This result does not cover \cite[Theorem 11.1]{DKV83} but could be viewed as a complementary variant. The estimate obtained independently by Matz and Templier \cite[Proposition 8.2]{MT}, and by Blomer and Pohl \cite[Theorem 2]{BP} establishes the uniform decay allowing $0\in\omega$. But the exponent in their estimate is not sharp compared to \cite[Theorem 11.1]{DKV83}. Finis and Matz obtained an estimate \cite[Proposition A.1]{FM} refining the result implied by \cite[Theorem 11.1]{DKV83} by allowing $H$ to approach $0\in \fa$ in a conic neighborhood $[0,1]\cdot \omega$. The results above are achieved on general groups but do not surpass our work on $\SL(3,\R)$ in this article. If we restrict their results to $\SL(3,\R)$ and compare them with Theorem \ref{est-sphfun}, our Theorem \ref{est-sphfun} is stronger. Moreover, our approach is independent and new. 

\subsection{Organization of the paper}

The paper is organized as follows. 

Section 2 covers preliminary results in differential geometry. Besides settling notations and conventions, a notion of transversality and an operation of projection are introduced. 

Section 3 collects basic properties of the phase function \eqref{eq:phase}. Most of them could be found in \cite[Section 1]{DKV83}. 

Section 4 is the bulk of the article where we classify the local normal forms case by case. In order to avoid overlapping work, the order of the cases is arranged for efficiency. 

Section 5 deals with the estimates of the oscillatory integral \eqref{eq:ost2} and \eqref{eq:ost1}. Applying a multivariate stationary phase estimate, the local estimates for oscillatory integral \eqref{eq:ost2} are immediate when the local normal forms are at our disposal. 

Section 6 gives an application of Theorem \ref{est-sphfun} to the study of asymptotic behaviors of eigenfunctions on locally symmetric spaces.

\subsection{Acknowledgements} 
I would like to express my sincere gratitude to my Ph.D. advisor Professor Simon Marshall. Without his generous support and genuine encouragement, this work would not be done. I am indebted to Professor Tobias Finis for bringing the result by Duistermaat \cite{D} to my attention. The main part of the paper comprised my Ph.D. thesis when I was a graduate student at the Department of Mathematics at University of Wisconsin-Madison. Section 6 was carried out when I am a postdoc at Beijing International Center for Mathematical Research. I would like to acknowledge my gratefulness to both institutes. This research was partially supported by the National Science Foundation Grant DMS-1501230 and DMS-1902173.

%


\section{Terminology and tools from differential geometry}




All our manifolds are assumed to be smooth (in the $\smo$ sense), without boundary and finite-dimensional. They need not be connected. By a submanifold of a manifold, we mean that the inclusion map is an embedding. A submanifold need be neither a closed subset nor a closed manifold. All our maps including functions and vector fields are assumed to be smooth unless otherwise stated.

The sense of germs is extensively adopted in our study of local properties although it is not literally stated. We consider germs of maps including germs of vector fields and germs of submanifolds. We do not use the equivalence relation but simply restrict the domains of the maps to open neighborhoods of a given point. If we use $\sX$ to denote a manifold and $\sX$ appears several times in the argument, it may refer to different open submanifolds every time, especially when we introduce maps locally. The readers shall keep this point in mind to avoid confusions about the well-definedness. Although it may be helpful to know how large the domains could be, their local existence would suffice our needs in this article.

Now we introduce our notions of transversality between submanifolds and tangent vectors (or vector fields) and the projection of tangent vectors (or vector fields) to submanifolds. These tools are crucial to our reduction arguments from the ambient manifold to its submanifolds.  

Let $\sX$ be a manifold and let $\sY$ be a submanifold of $\sX$ with codimension $d$.  For $y\in \sY$, we identify the tangent space $\tg{\sY}{y}$ with a subspace of the tangent space $\tg{\sX}{y}$. So the tangent bundle $\tg{\sY}{}$ is identified with a submanifold of the tangent bundle $\tg{\sX}{}$. Let $f:\sX\to \C$. We denote its restriction to $\sY$ by $\res{f}{\sY}:\sY\to\C$. Let $v:\sX\to \tg{\sX}{}$ be a vector field on $\sX$. Its restriction to $\sY$, denoted by $\res{v}{\sY}:\sY\to \tg{\sX}{}$ is a section of $\tg{\sX}{}$ over $\sY$. If $v(y)\in \tg{\sY}{y}$ for all $y\in\sY$, then we call $v$ tangent to $\sY$ and can identify $\res{v}{\sY}$ with a vector field on $\sY$.

\begin{definition} \label{def-trans}
Let $y_0$ be a given point in $\sY$ and $X_1, X_2, \dots, X_d\in \tg{\sX}{y_0}$. The collection of tangent vectors $\{X_1, X_2, \dots, X_d\}$ is called transverse to $\sY$ at $y_0$ if
	\begin{equation} \label{eq:2-1}
		\tg{\sX}{y_0}=\tg{\sY}{y_0}+\R X_1+\R X_2+\dots+\R X_d,
	\end{equation}
here the right side is a direct sum. 
\end{definition}
 
\begin{definition} \label{def-proj}
Let the hypothesis of transversality in Definition \ref{def-trans} be fulfilled. For $X\in \tg{\sX}{y_0}$, its projection to $\sY$ with respect to $\{X_1,X_2,\dots,X_d\}$, denoted by $\pr{X}{\sX}{\sY}\in \tg{\sY}{y_0}$ is defined to be the unique tangent vector corresponding to the $\tg{\sY}{y_0}$-component in the summation, that is,
	\[X=\pr{X}{\sX}{\sY}+a_1X_1+a_2X_2+\dots+a_d X_d,\]
here $a_1,a_2,\dots,a_d\in \R$.
\end{definition}
It is needless to say that $X=\pr{X}{\sX}{\sY}$ when $X\in \tg{\sY}{y_0}\subseteq \tg{\sX}{y_0}$. The following lemma gives an alternative definition of the projection of tangent vectors to submanifolds. 
\begin{lemma} \label{lem-equi}
Let the hypothesis of transversality in Definition \ref{def-trans} be fulfilled. Let $\kappa:\sX\to \sY\times \R^d$ be a local fibration at $y_0$, that is, $\kappa(y)=(y,0),y\in \sY$. Suppose $\kappa$ satisfies \[d\kappa(\R X_1+\R X_2+\dots+\R X_d)=\tg{\R^d}{0}\subseteq \tg{\sY\times \R^d}{(y_0,0)}.\]
Let $\pi:\sX\to \sY$ be the composition of $\kappa$ and the projection $\sY\times \R^d\to \sY$. Then we have
	\[\pr{X}{\sX}{\sY}=d\pi(X),\quad X\in \tg{\sX}{y_0},\]
here $\pr{X}{\sX}{\sY}$ is the projection of $X$ to $\sY$ with respect to $\{X_1,\dots,X_d\}$.
\end{lemma}

The following handy lemma indicates how we will take advantage of the notions above.

\begin{lemma} \label{lem-id}

Let the hypothesis of transversality in Definition \ref{def-trans} be fulfilled. Let $f$ be a $\smo$ function on $\sX$ and $\res{f}{\sY}$ its restriction to $\sY$. If \[X_1f=X_2f=\dots=X_df=0,\] 
then we have 
	\[\pr{X}{\sX}{\sY}(\res{f}{\sY})=Xf, \quad X\in \tg{\sX}{y}.\] 
\end{lemma}

The statements from Definition \ref{def-trans} to Lemma \ref{lem-id} are made in the pointwise manner. However, it is clear that we could employ them in the `local' manner. Let $v_1,\dots,v_d$ be vector fields on $\sX$ or  defined in a neighborhood of $y_0$ in $\sX$. We say that the collection of these vector fields is transverse to $\sY$ at $y_0$ if the collection of their values at $y_0$, $\{v_1(y_0),\dots,v_d(y_0)\}$ is transverse to $\sY$ at $y_0$. Then it is clear that there exists a neighborhood $U$ of $y_0$ in $\sY$ such that this transversality holds for all $y\in U$. Let $v$ be a vector field on $\sX$. Applying Definition \ref{def-proj} to $v(y)$ for all $y\in U$,  we get a map $U\to \tg{U}{}: y\to \pr{v(y)}{\sX}{\sY}$. It is easy to see that this map is $\smo$ and could be identified with a vector field on the submanifold $U$. We denote it by $\pr{v}{\sX}{\sY}$. Since we are concerned with local properties, it is harmless to replace $\sY$ with $U$ to restrict the domain. As our argument involves a lot of such manipulations, we would seldom specify the neighborhood $U$ and still use $\sY$ to denote it for saving notations. The readers shall be aware of our changes of domains.

\begin{remark}
Let us give an explanatory remark regarding our definition of the transversality.  One well-known definition of transversality is defined for two submanifolds cf. Hirsch \cite{Hi}. Let $\sY,\sZ$ be two submanifolds of $\sX$. For $y_0\in \sY\cap\sZ$, $\sY$ and $\sZ$ intersect transversally at $y_0$ if $\tg{\sX}{y_0}=\tg{\sY}{y_0}+\tg{\sZ}{y_0}$, here the sum need not be direct. In this article, we do not study  the intersection of submanifolds. The following is the typical scenario we face. We come up with a submanifold $\sY$ as a level set and have some vector fields on hand. We try to obtain a local fibration at $y_0\in \sY$, $\kappa:\sX\to \sY\times\R^d$. This local fibration yields a family of submanifolds $\sZ_y:=\kappa^{-1}(\{y\}\times \R^d)$'s which are parametrized smoothly by $y\in \sY$. We have $\tg{\sX}{y}=\tg{\sY}{y}+\tg{\sZ_y}{y}$, here the sum is direct.
\end{remark}

Let $\sX'$ be another manifold. Let us consider the Cartesian product $\sX\times \sX'$. We would like to view it as a fiber bundle with the base space $\sX$ and the fiber $\sX'$. Basically, we want $\sX$ to be the space of parameters, $\sX'$ the space within which we take the manipulations like the integration. The tangent bundle $\tg{\sX\times \sX'}{}$ is canonically isomorphic to the Cartesian product $\tg{\sX}{}\times \tg{\sX'}{}$. Then the vertical bundle is defined to be the subbundle $\sX\times \tg{\sX'}{}$. A tangent vector is called vertical if it belongs to the vertical bundle. Let $\sY$ be a submanifold of $\sX\times\sX'$. A section of the tangent bundle $\tg{\sX\times \sX'}{}$ over $\sY$ is vertical if it is a section of the vertical bundle $\sX\times \tg{\sX'}{}$. Furthermore, if this vertical section $v$ is a vector field on $\sY$, then its integral curves are along the fibers. One kind of vertical vector fields comes from the canonical extension of  a vector field on $\sX'$ to $\sX\times \sX'$. Let $v$ be a vector field on $\sX'$. By a canonical extension of $v$, we mean a vector field $\tilde{v}$ on $\sX\times\sX'$ such that for $(x,x')\in \sX\times \sX'$,
	\[\tilde{v}(x,x'):=v(x')\in \tg{\sX'}{x'}=\{x\}\times \tg{\sX'}{x'}\subseteq \tg{\sX\times\sX'}{(x,x')}.\] Finally, we name a simple fact. Let $X\in \{x\}\times \tg{\sX'}{x'}$, let $f_1\in \smo(\sX,\C)$ and let $f_2\in \smo(\sX\times \sX',\C)$. Then we have $X(f_1\cdot f_2)=f_1(x)\cdot Xf_2$, namely, we could pull the factor $f_1$ outside the differentiation.
	
\section{Properties of the function $f_{H, H'}$}

For $Y,Y'\in\fg$, we consider the function
	\[f_{Y, Y'}:\rK\to \R,\quad k\to B(Y,\Ad(k)Y').\]
On most occasions, we consider $Y,Y'\in \fa$ or $Y,Y'\in\fs$.

In this section, we give a brief account of properties of $f_{Y,Y'}$. Most of them could be found in \cite[Section 1]{DKV83}. Here, they are included to make the article more self-contained and adapted for our needs. Results in \cite{DKV83} are stated for general semisimple Lie groups and we only need their validity for the specific case $\rG=\SL(3,\R)$ and $\rK=\SO(3)$.

Due to the properties of the adjoint representation and the Killing form, we know that $f_{Y,Y'}$ is $\rK_Y$-left-invariant and $\rK_{Y'}$-right-invariant. Here $\rK_Y$, resp. $\rK_{Y'}$ is the centralizer of $Y$, resp. $Y'$ in $\rK$, exactly the stabilizer subgroup of $\rK$ with respect to $Y$, resp. $Y'$.

For $k_0\in \rK$, we identify $\tg{\rK}{k_0}$ with $\fk=\tg{\rK}{e}$ by the left translation. Then for $X\in \tg{\rK}{k_0}=\fk$, we have
	\[Xf_{Y,Y'}=B(Y,\Ad(k_0)[X,Y'])=f_{Y,[X,Y']}(k_0).\]
	
For $X\in\fk$, let $\inv{X}{L}$ be the left invariant vector field on $\rK$ associated to $X$ and let $\inv{X}{R}$ be the right invariant vector field associated to $X$. Then we have that for $k\in\rK$,
	\begin{equation}\label{eq:3-1}
		[\inv{X}{L}f_{Y,Y'}](k)=f_{Y,[X,Y']}(k),\quad [\inv{X}{R}f_{Y,Y'}](k)=f_{[Y,X],Y'}(k).
	\end{equation}
	Here we use the fact that the Killing form is invariant under the adjoint action. Equation \eqref{eq:3-1} tells that if we differentiate $f_{Y,Y'}$ with respect to a few left or right invariant vector fields, it amounts to a recursion of calculating Lie brackets.
	
	Now we focus on the situation $Y=H\in\fa$ and $Y'=H'\in\fa$. For $X\in\fk$, we write 
		\begin{equation}\label{eq:rootdecom}
			X=\sum_{\alpha\in \Rr^+} X_{\alpha}+\theta X_{\alpha},\quad X_\alpha\in\fg_\alpha.
		\end{equation}
For the sake of convenience, we put $\fk_\alpha=\fk\cap(\fg_\alpha+\fg_{-\alpha})$,  $\fs_\alpha=\fs\cap(\fg_{\alpha}+\fg_{-\alpha})$ for $\alpha\in \Rr$. We have  $X_\alpha+\theta X_\alpha\in \fk_\alpha$ and $X_\alpha-\theta X_\alpha\in \fs_\alpha$. 

With the decompostion \eqref{eq:rootdecom} of $X$, Equation \eqref{eq:3-1} yields that for $k\in\rK$,
	\begin{equation}
		[\inv{X}{L}f_{H,H'}](k)=\sum_{\alpha\in\Rr^+} -\alpha(H')f_{H,X_\alpha-\theta X_\alpha}(k),
	\end{equation}
		\begin{equation}
		[\inv{X}{R}f_{H,H'}](k)=\sum_{\alpha\in\Rr^+} \alpha(H)f_{X_\alpha-\theta X_\alpha,H'}(k).
	\end{equation}
Thus if we take $X\in\fk_\alpha$ for some $\alpha\in \Rr^+$, we could extract a factor $\alpha(H')$, resp. $\alpha(H)$ from differentiating $f_{H,H'}:\rK\to \R$ with respect to $\inv{X}{L}$, resp. $\inv{X}{R}$. We point out that this plain fact is going to play a significant role in our investigation of the stability of the singularities. 

\begin{remark}
Generally speaking,  when $H$ turns singular, that is, $\alpha(H)$ goes to 0 for some $\alpha\in \Rr^+$, we could use the right invariant vector field $\inv{X}{R}$ for $X\in \fk_{\alpha}$ to resolve the singularity coming from $H$. Accordingly, when $H'$ turns singular, we could use the left invariant vector field to resolve singularity coming from $H'$. This strategy is very helpful in the study of the singularities and their stability. However, it is not sufficient to handle all the cases for the following reason. Our problem concerns the singularities when $H$ and $H'$ both vary. Using left or right invariant vector field could only resolve the singularity from one side. In order to deal with the singularities arising from the both sides simultaneously, we shall incorporate facts in Lie groups and Lie algebras. Some of these facts are easily observed in the particular case $\rG=\SL(3,\R)$. 
\end{remark}

To apply the method of stationary phase to the estimate of the oscillatory integral \eqref{eq:ost2}, we shall locate the critical set of $f_{H,H'}$. 

\begin{lemma}[\cite{DKV83} Proposition 1.2]
The critical set of $f_{H,H'}$ is equal to 
	\[\bigcup_{s\in \rW} \rK_H s\rK_{H'},\]
here $\rK_H s\rK_{H'}$ is the double coset $\{k_1k_sk_2|k_1\in \rK_H,k_s\in s,k_2\in \rK_{H'}\}$.

\end{lemma}

It is clear that $\rM\subseteq \rK_H$ for any $H\in \fa$. In fact, \cite[Proposition 1.2]{DKV83} tells that
	\[\rK_{H}=\rM\rK_H^0=\rK_{H}^0\rM,\]
here $\rK_{H}^0$ is the connected component that containing the identity element of $\rK_H$. For $s\in \rW, H\in\fa$, we have $\rK_Hs=s\rK_{s^{-1}H}$. 

Fix $s\in\rW$, $\rK_Hs\rK_{H'}$ viewed as a topological subspace of $\rK$ may not be connected. Neverthelss, its connected components could be identified with closed submanifolds of $\rK$ and are diffeomorphic to each other via either left or right translations. Thus it makes sense to view $\rK_{H}s\rK_{H'}$ as a submanifold of $\rK$. As it consists of critical points of $f_{H,H'}$, we call it a critical manifold of $f_{H,H'}$. 

Let $s\in\rW$ vary. \cite[Proposition 1.3]{DKV83}  calculates the dimension of the critical manifold $\rK_{H}s\rK_{H'}$ which depends on $H,H'$ and $s$. It tells that the critical set of $f_{H,H'}$ is a union of critical manifolds which may have different dimensions. \cite[Proposition 1.4]{DKV83}   computes the Hessian of $f_{H,H'}$ at the critical points and \cite[Corollary 1.5]{DKV83} tells that $f_{H,H'}$ is a Morse-Bott function on $\rK$.

In order to establish uniform estimate for the oscillatory integral \eqref{eq:ost2}, we shall study the deformation of $f_{H,H'}$ when $H,H'$ vary. As explained in \cite[Remark 1.6]{DKV83}, the dependence of the critical set of $f_{H,H'}$ on the parameters $H,H'$ has a highly nongeneric rigidity. When $H,H'$ vary around the root hyperplanes, the dimension of the critical manifolds could change abruptly. This unstable feature in the qualitative aspect leads to the challenge in establishing a uniform and sharp estimate for the oscillatory integral in the quantitative aspect.

Now, we look closer at our specific case $\rG=\SL(3,\R)$. The centralizer of $\fa$ in $\rK$, $\rM$ and the normalizer of $\fa$ in $\rK$, $\rM'$ are discrete subgroups of $\rK$. If $H,H'$ are both regular, then the critical set of $f_{H,H'}$ is equal to $\rM'$ and $f_{H,H'}$ is a Morse function on $\rK$. For $H\in \fa$ singular and not 0, $\rK_H$ is a Lie subgroup of $\rK$ with dimension 1. It is obvious that $\rK_H^0$ is abelian. For $H,H'\in\fa$ not 0, $s\in \rW$, $\rK_H s\rK_{H'}$ as a critical manifold of $f_{H,H'}$  has dimension at most 2.

\section{Local normal forms}
Fix $H,H'\in\fa$. Morse lemma or Morse-Bott lemma gives the local normal forms of $f_{H,H'}$ at the critical points. Namely, under a suitable coordiniate system, $f_{H,H'}$ is expressed as a quadratic form plus a constant.

Our problem concerns the deformation of $f_{H,H'}$ as the parameters $H,H'$ vary. We seek a family of local coordinate systems depending on parameters $H,H'$ smoothly such that $f_{H,H'}$ is uniformized into a normal form. 

As $\fa\times \fa$ has dimension 4, we need $4$ real numbers to parametrize it. Due to the linearity of the Killing form and other considerations, we make the following setting.

Let $H_1,H'_1\in \fa$ with $\len{H_1}=\len{H'_1}=1$. Take $H_2,H'_2\in\fa$ such that $\{H_1,H_2\}$ and $\{H'_1,H'_2\}$ are orthonormal bases of $\fa$. Consider the maps 
	\[\R^2\to \fa,\quad (b_1, b_2)\to b_1(H_1+b_2H_2),\]
	\[\R^2\to \fa,\quad (c_1, c_2)\to c_1(H'_1+c_2H'_2).\]
They are not surjective since $\R H_2$ or $\R H'_2$ are not included in the range. If the domain is restricted to $\R\times (-\varepsilon,\varepsilon)$ for $\varepsilon>0$, its image is a conic subset (a sector) in $\fa$. With different orthonormal bases, we have different conic subsets of $\fa$ which cover the entire $\fa$.

Put $S_1=\R\times \R\times \rK$ which is viewed as a fiber bundle with the base space $\R\times\R$ and fiber $\rK$. We put $f_1:S_1\to \R$,
\[f_1(b_2, c_2; k)=B(H_1+b_2H_2,\Ad k(H'_1+c_2H'_2)).\]

Our task is to get local normal forms of $f_1:S_1\to \R$ at $(0,0;k_0)$ for any critical point $k_0$ of $f_{H_1,H'_1}$ and for all $H_1,H'_1\in \fa$ with $\len{H_1}=\len{H'_1}=1$. We first study the difficult case when $H_1,H'_1$ are both singular, then move to the simpler case when either $H_1$ or $H'_1$ is regular. The advantage of this order is that the preceding work overlaps the succeeding one.

We remark that concerning the stability of the singularity we deal with, Morse lemma with parameters does not apply since the nondegenercy condition does not hold. However, the readers should be able to find in our argument the spirit of the classical proof of Morse lemma by reduction. 

\subsection{$H_1$ and $H'_1$ are both singular}
Suppose that $H_1$ and $H'_1$ are both singular. If $k_0\in \rK$ is a critical point of $f_{H_1,H'_1}$, then $k_0$ is contained in a critical manifold $\rK_{H_1}s\rK_{H'_1}$ for some $s\in \rW$. The critical manifold $\rK_{H_1}s\rK_{H'_1}$ has dimension either 1 or 2 at $k_0$. In this section, we first discuss the situation when it has dimension 1, then discuss the situation when it has dimension 2.

The critical manifold $\rK_{H_1}s\rK_{H'_1}$has dimension 1, if and only if $\rK_{s^{-1}H_1}=\rK_{H'_1}$. It is also equivalent to $s^{-1}H_1=\pm H'_1$ as we have assumed that $\len{H_1}=\len{H'_1}=1$. We have $\rK_{H_1}s\rK_{H'_1}=\rK_{H_1}s=s\rK_{H'_1}$. We consider the partition 
	\[\rK_{H_1}s\rK_{H'_1}=(\rK_{H_1}s\rK_{H'_1}\cap \rM')\cup (\rK_{H_1}s\rK_{H'_1}\backslash \rM').\]
	We first work on the local normal forms of $f_1$ for $k_0\in \rK_{H_1}s\rK_{H'_1}\cap \rM'$, then for $k_0\in \rK_{H_1}s\rK_{H'_1}\backslash \rM'$.

\begin{theorem} \label{thm-n1}
Assume $k_0$ is contained in the critical manifold $\rK_{H_1}s\rK_{H'_1}$ which has dimension 1. We arrange the positive roots $\alpha_1,\alpha_2,\alpha_3$ in the way such that $\alpha_3$ is the unique positive root with $s\alpha_3(H_1)=\alpha_3(H'_1)=0$.

If $k_0\in \rK_{H_1}s\rK_{H'_1}\cap \rM'$, then there exists a local coordinate system at $(0,0;k_0)$, $\kappa:S_1\to \R^5$ which preserves the parameters, that is, 
\[\kappa(b_2,c_2;k)=(b_2,c_2;\cdot).\] 
$\kappa$ satisfies
	\[
		\begin{split}
			f_1\circ \kappa^{-1}(b_2, c_2; \cx_1, \cx_2, \cx_3)&=f_1(b_2, c_2; k_0)-\frac{1}{2}s\alpha_1(H)\alpha_1(H')\cx_1^2\\
			&-\frac{1}{2}s\alpha_2(H)\alpha_2(H')\cx_2^2-\frac{1}{2}\epsilon\cdot s\alpha_3(H)\alpha_3(H')\cx_3^2,
		\end{split}
	\]

	here $H=H_1+b_2H_2, H'=H'_1+c_2H'_2$ and $\epsilon$ is a constant $\pm 1$ which is determined explicitly. Besides, $\kappa$ always maps $(b_2,c_2;k_0)$ to $(b_2,c_2;0,0,0)$.
\end{theorem}

Under the local coordinate system $\kappa$, $f_1\circ \kappa^{-1}$ becomes a quadratic polynomial. The local normal form refers to that. In the formula, $s\alpha_1(H)$, $s\alpha_2(H)$, $\alpha_1(H')$, $\alpha_2(H')$ stay away from 0 so they could be replaced by $\pm1$ after a trivial modification of $\kappa$. They are included for the purpose of consistency.

\begin{proof}
According to our arrangement of the positive roots,  $s\alpha_1(H_1)$, $s\alpha_2(H_1)$, $\alpha_1(H'_1)$, $\alpha_2(H'_1)$ are all not 0. 

For $k_0\in \rK_{H_1}s\rK_{H'_1}\cap \rM'=\rK_{H_1}s\cap \rM'$, $k_0=k_1k_s$ with $k_1\in \rM'_{H_1},k_s\in s$. So $\Ad (k_0^{-1})H_1=s^{-1}H_1$. As $\Ad (k_1^{-1})H_1=H_1$, $\Ad (k_1^{-1})H_2=\pm H_2$. Let $\epsilon=\Ad (k_1^{-1})H_2/H_2$. Note $\R H_2=\R (s\alpha_3)^{\vee}$.

Let $\{X_1,X_2,X_3\}$ be an orthonormal basis of $\fk$ such that $X_i\in \fk_{\alpha_i}$ for $i=1,2,3$. We take the left invariant vector fields $\inv{X_1}{L},\inv{X_2}{L},\inv{X_3}{L}$ on $\rK$. Let $v_{1,1},v_{1,2},v_{1,3}$ be their canonical extensions to $S_1=\R\times \R\times \rK$.

We construct the first local coordinate $\cx_1:S_1\to \R$. We give a list of manipulations together with the facts supporting the manipulations.


(1a). Consider the function $v_{1,1}f_1:S_1\to \R$,
	\[\begin{split}
		&\quad[v_{1,1}f_1](b_2,c_2;k)\\
		&=B(H_1+b_2 H_2,\Ad k[X_1,H'_1+c_2H'_2])\\
						   &=-\alpha_1(H'_1+c_2H'_2)B(H_1+b_2H_2,\Ad k(X_{1,\alpha_1}-\theta X_{1,\alpha_1})),	
	\end{split}\]
	We see that $v_{1,1}f_1$ is a product of two functions on $S_1$. We put $\mu_{1,1},g_{1,1}:S_1\to \R$ by
		\begin{equation}\label{eq:1a1}
			\mu_{1,1}(b_2,c_2;k):=-\alpha_1(H'_1+c_2H'_2),
		\end{equation}
		\begin{equation}\label{eq:1a2}
		g_{1,1}(b_2,c_2;k):=B(H_1+b_2H_2,\Ad k(X_{1,\alpha_1}-\theta X_{1,\alpha_1})).
		\end{equation}
	Note that $\mu_{1,1}$ only depends on $c_2$ and $g_{1,1}$ does not depend on $c_2$. The point of factoring out terms like $\mu_{1,1}$ is that if $b_2,c_2$ are fixed, these terms are  treated as constants, but when $b_2,c_2$ vary, they may carry the singularity of $b_2,c_2$. Here $\mu_{1,1}$ does not carry the singularity for $\alpha_1(H'_1)\neq 0$.
\medskip 

(1b). We show that the function $g_{1,1}:S_1\to \R$ vanishes at $(b_2,c_2;k_0)\in S_1$ and $g_{1,1}$ is a submersion at $(0,0;k_0)$. 

Evaluate $g_{1,1}$ at $(b_2,c_2;k_0)$,
	\[g_{1,1}(b_2,c_2;k_0)=B(\Ad k_0^{-1}(H_1+b_2H_2),X_{1,\alpha_1}-\theta X_{1,\alpha_1}).\]
$k_0\in \rM'$ so $\Ad k_0^{-1}(H_1+b_2H_2)\in \fa$. As $X_{1,\alpha_1}-\theta X_{1,\alpha_1}\in \fq$, $g_{1,1}(b_2,c_2;k_0)$ is 0.

To show that $g_{1,1}$ is a submersion at $(0,0;k_0)$, we show $[v_{1,1}g_{1,1}](0,0;k_0)\neq 0$. 
	\[\begin{split}
		[v_{1,1}g_{1,1}](0,0;k_0)&=B(H_1,\Ad k_0[X_1,X_{1,\alpha_1}-\theta X_{1,\alpha_1}])\\
						     &=B([\Ad k_0^{-1}H_1,X_1],X_{1,\alpha_1}-\theta X_{1,\alpha_1})\\
						     &=s\alpha_1(H_1)\neq 0.
	\end{split}\]
\medskip

(1c).  As $g_{1,1}$ is a submersion at $(0,0;k_0)$, the implicit function theorem tells that the level set $g_{1,1}^{-1}(\{0\})$ is locally a submanifold of $S_1$ through $(0,0;k_0)$ and we denote it by $S_2$. $S_2$ has codimension 1 in $S_1$. According to the fact that $g_{1,1}(b_2,c_2;k_0)=0$ in (1b), the subset $\{(b_2,c_2,k_0)\in S_2\}$ could be identified with a 2-dimension submanifold of $S_2$. We also know that $v_{1,1}$ is transverse to $S_2$ at $(0,0;k_0)$ since $[v_{1,1}g_{1,1}](0,0;k_0)\neq 0$. Hence we can assume that $v_{1,1}$ is transverse to $S_2$ everywhere by replacing $S_2$ with a sufficiently small (open) neighborhood of $(0,0;k_0)$ in $S_2$.

\medskip

(1d). The flow generated by $v_{1,1}$ which is a map $S_1\times \R\to S_1$ induces a local fibration at $(0,0;k_0)$, 
\[\kappa_1:S_1\to S_2\times \R, \quad s_1\to (s_2,t_1).\] 
$\kappa_1$ satisfies $d\kappa_1(v_{1,1})=\partial/\partial t_1$ and $\kappa_1(s_2)=(s_2,0)$ for $s_2\in S_2$. Let $\pi_1:S_1\to S_2$ be the composition of $\kappa_1$ and the projection $S_2\times \R\to S_2$. It is obvious that $d\pi_1(v_{1,1})=0$. We point out that the local fibration $\kappa_1$ preserves the parameters, that is, if $\kappa_1(s_1)=(s_2,t_1)$ and $s_1=(b_2,c_2;k)$, then $s_2=(b_2,c_2;\cdot)$. It is true because the vector field $v_{1,1}$ is vertical so the flow is along the fiber $\rK$.

We consider the function $(f_1-f_1\circ \pi_1)\circ \kappa_1^{-1}:S_2\times \R\to \R$. We show it vanishes along $S_2\times \{0\}$ up to the first order, that is, for $(s_2,0)\in S_2\times \R$,
	\[[(f_1-f_1\circ \pi_1)\circ \kappa_1^{-1}](s_2,0)=0,\]
	\[\partial_{t_1}[(f_1-f_1\circ \pi_1)\circ \kappa_1^{-1}](s_2,0)=0.\]
It vanishes along $S_2\times\{0\}$ because for $s_2\in S_2$, we have $\kappa_1^{-1}(s_2,0)=s_2$ and $\pi_1(s_2)=s_2$.

In order to show that the first order derivative in variable $t_1$ vanishes along $S_2\times \{0\}$, we rewrite
	\[\begin{split}
		&\quad\partial_{t_1}[(f_1-f_1\circ \pi_1)\circ \kappa_1^{-1}](s_2,t_1)\\
			&=\partial_{t_1}[f_1\circ \kappa_1^{-1}](s_2,t_1)\\
			&=[(v_{1,1}f_1)\circ \kappa_1^{-1}](s_2,t_1)\\
			&=[\mu_{1,1}\circ\kappa_1^{-1}](s_2,t_1)\cdot [g_{1,1}\circ\kappa_1^{-1}](s_2,t_1),
	\end{split}\]
where $(s_2,t_1)\in S_2\times \R.$ By the definition of $S_2$, we know the factor $g_{1,1}\circ \kappa_1^{-1}$ vanishes on $S_2\times \{0\}$, so does $\partial_{t_1}[(f_1-f_1\circ \pi_1)\circ \kappa_1^{-1}]$.
	
	As the fibration $\kappa_1$ preserves the parameters and $\mu_{1,1}$ only depends on $c_2$, we have $[\mu_{1,1}\circ\kappa_1^{-1}](s_2,t_1)=\mu_{1,1}(s_2)$. Thus we can replace $[\mu_{1,1}\circ\kappa_1^{-1}](s_2,t_1)$ by $\mu_{1,1}(s_2)$ in the expression of the first order derivative above, here $\mu_{1,1}(s_2)$ should be viewed as a function on $S_2\times \R$ but is independent of $t_1$.
	
	We also consider the second order derivative $\partial_{t_1}^2[(f_1-f_1\circ \pi_1)\circ \kappa_1^{-1}]$ and show $\partial_{t_1}^2[(f_1-f_1\circ \pi_1)\circ \kappa_1^{-1}]((0,0;k_0),0)$ not 0. Then we ensure that the second order derivative stays away from 0 everywhere by restricting to a sufficiently small (open) neighborhood of $((0,0;k_0),0)$ in $S_2\times \R$. We write
		\[\begin{split}
			&\quad\partial_{t_1}^2[(f_1-f_1\circ\pi_1)\circ \kappa_1^{-1}](s_2,t_1)\\
			&=\partial_{t_1}[\mu_{1,1}(s_2)\cdot g_{1,1}\circ \kappa_1^{-1}](s_2,t_1)\\
			&=\mu_{1,1}(s_2)\cdot [(v_{1,1}g_{1,1})\circ\kappa_1^{-1}](s_2,t_1),
		\end{split}\]
	which tells that it is a product of two functions. 
	
	Since $\mu_{1,1}(0,0;k_0)=-\alpha_1(H'_1)\neq 0$ and $[v_{1,1}g_{1,1}](0,0;k_0)=s\alpha_1(H_1)\neq 0$, the second order derivative at $((0,0;k_0),0)$ is not 0.
	
\medskip	
	
(1e). We apply the Taylor's formula to $(f_1-f_1\circ \pi_1)\circ \kappa_1^{-1}:S_2\times \R\to \R$ in the variable $t_1$,
		\[\begin{split}
			&\quad [(f_1-f_1\circ \pi_1)\circ\kappa_1^{-1}](s_2,t_1)\\
			&=\frac{1}{2}t_1^2\int_0^1 2(1-z)\partial_{t_1}^2[(f_1-f_1\circ \pi_1)\circ\kappa_1^{-1}](s_2,t_1z)dz\\
				&=\frac{1}{2}\mu_{1,1}(s_2)t_1^2\int_0^1 2(1-z)[(v_{1,1}g_{1,1})\circ\kappa_1^{-1}](s_2,t_1z)dz.
		\end{split}\]
		We insert $s\alpha_1(H_1+b_2H_2)$ into the right side of the equation
		\begin{eqnarray*}
		&&\frac{1}{2}\mu_{1,1}(s_2)s\alpha_1(H_1+b_2H_2)t_1^2\frac{1}{s\alpha_1(H_1+b_2H_2)}\\
		&&\quuad\quad\quad \cdot \int_0^1 2(1-z)[(v_{1,1}g_{1,1})\circ\kappa_1^{-1}](s_2,t_1z)dz,
		\end{eqnarray*}
		here 	$s\alpha_1(H_1+b_2H_2)$ should be viewed as a function on $S_2\times \R$ which only depends on the $b_2$ component of $s_2$. The division by $s\alpha_1(H_1+b_2H_2)$ makes sense because $s\alpha_1(H_1)\neq 0$ implies that $s\alpha_1(H_1+b_2H_2)$ stays away from 0 when $|b_2|$ is sufficiently small or we replace $S_2$ by a sufficiently small neighborhood of $(0,0;k_0)$ in $S_2$.
		
		For the purpose of uniformization, we introduce a local diffeomorphism at $((0,0;k_0),0)$, $S_2\times \R\to S_2\times \R, (s_2,t_1)\to (s_2,\ct_1)$, by setting
			\[\ct_1(s_2,t_1)=t_1\Big[\frac{1}{s\alpha_1(H_1+b_2H_2)}\int_0^1 2(1-z)[(v_{1,1}g_{1,1})\circ\kappa_1^{-1}](s_2,t_1z)dz\Big]^{1/2},\]
		here taking the square root makes sense because at $((0,0;k_0),0)\in S_2\times\R$, the term inside the square root is equal to
			\[\frac{1}{s\alpha_1(H_1)}\int_0^1 2(1-z)[v_{1,1}g_{1,1}](0,0;k_0)dz=\int_0^1 2(1-z)dz=1,\]
		which tells the existence of a small neighborhood where the term inside the square root remains positive.
	
		Let $\cka_1:S_1\to S_2\times \R, s_1\to (s_2,\ct_1)$ be the (local) composition of $\kappa_1$ and the map above. It is also a local fibration at $(0,0;k_0)$. We have that for $(b_2,c_2;k)\in S_1$,
			\[[f_1-f_1\circ\pi_1](b_2,c_2;k)=-\frac{1}{2}s\alpha_1(H_1+b_2H_2)\alpha_1(H'_1+c_2H'_2)\ct_1^2,\]
		which could be rewritten as
			\[f_1(b_2,c_2;k)=f_1\circ\pi_1(b_2,c_2;k)-\frac{1}{2}s\alpha_1(H_1+b_2H_2)\alpha_1(H'_1+c_2H'_2)\ct_1^2.\]
		The first local coordinate $\cx_1:S_1\to \R$ is the $\ct_1$ component.
		

		Next, we construct the second local coordinate $\cx_2:S_1\to \R$. We list the manipulations and facts as well.
		


(2a). We know that $v_{1,1}$ is transverse to $S_2$ everywhere in (1c). Thus we can take the projections of $v_{1,2},v_{1,3}$ to $S_2$ with respect to $v_{1,1}$ and we denote the projections by $v_{2,2},v_{2,3}$. They are vertical vector fields on $S_2$. Precisely, under the identification of $\tg{S_2}{}$ with a submanifold of $\tg{S_1}{}$, $v_{2,2}$ and $v_{2,3}$ are vertical pointwisely.
		
\medskip


(2b). Consider the function $v_{1,2}f_1:S_1\to \R$. As in (1a), it could be written as a product of two functions $\mu_{1,2},g_{1,2}:S_1\to \R$,
			\[[v_{1,2}f_1](b_2,c_2;k)=\mu_{1,2}(b_2,c_2;k)g_{1,2}(b_2,c_2;k),\]
			\[\mu_{1,2}(b_2,c_2;k):=-\alpha_2(H'_1+c_2H'_2),\]
			\[g_{1,2}(b_2,c_2;k):=B(H_1+b_2H_2,\Ad k(X_{2,\alpha_2}-\theta X_{2,\alpha_2})).\]

			We remind that $\alpha_2(H'_1)\neq 0$ according to our choice of $\alpha_2$. Let $f_2:S_2\to \R$ be the restriction of $f_1:S_1\to \R$ to $S_2$. Since $v_{1,1}f_1$ vanishes on $S_2$, Lemma \ref{lem-id} tells that 
				\[[v_{2,2}f_2](b_2,c_2;k)=[v_{1,2}f_1](b_2,c_2;k),\quad (b_2,c_2;k)\in S_2.\]
			Let $\mu_{2,2},g_{2,2}:S_2\to \R$ be the restrictions of $\mu_{1,2},g_{1,2}:S_1\to \R$ to $S_2$. The function $v_{2,2}f_2:S_2\to \R$ is the product of two functions $\mu_{2,2},g_{2,2}:S_2\to \R$.

\medskip
	
(2c). We show that the function $g_{2,2}:S_2\to \R$ vanishes at $(b_2,c_2;k_0)\in S_2$ and $g_{2,2}$ is a submersion at $(0,0;k_0)$. $g_{2,2}(b_2,c_2;k_0)$ is always 0 for the same reason in (1b). To prove the later, we shall be more careful. We show $[v_{2,2}g_{2,2}](0,0;k_0)$ is not 0 indirectly. We first show that  $[v_{1,1}g_{1,2}](0,0;k_0)$ is 0. We have
		\[[v_{1,1}g_{1,2}](0,0;k_0)=B([\Ad k_0^{-1}(H_1),X_1],X_{2,\alpha_2}-\theta X_{2,\alpha_2}).\]
	$\Ad k_0^{-1} (H_1)\in \fa$, so $[\Ad k_0^{-1}(H_1),X_1]\in \fs_{\alpha_1}$. While $X_{2,\alpha_2}-\theta X_{2,\alpha_2}\in \fs_{\alpha_2}$, so $[v_{1,1}g_{1,2}](0,0;k_0)$ is 0.
	
	Then we show $[v_{1,2}g_{1,2}]$ is not 0. We have
		\[\begin{split}[v_{1,2}g_{1,2}](0,0;k_0)&=B([\Ad k_0^{-1}(H_1),X_2],X_{2,\alpha_2}-\theta X_{2,\alpha_2})\\
			&=s\alpha_2(H_1)\neq 0.
		\end{split}\]
	Then Lemma \ref{lem-id} tells that $[v_{2,2}g_{2,2}](0,0;k_0)$ is equal to $[v_{1,2}g_{1,2}](0,0;k_0)$, so it is not 0.
	
\medskip

(2d). As $g_{2,2}:S_2\to \R$ is a submersion at $(0,0;k_0)$, the implicit function theorem tells that the level set $g_{2,2}^{-1}(\{0\})$ is locally a submanifold of $S_2$ through $(0,0;k_0)$ and let us denote it by $S_3$. $S_3$ has codimension 1 in $S_2$. (2c) tells that $g_{2,2}(b_2,c_2;k_0)$ is always 0, so the subset $\{(b_2,c_2;k_0)\in S_3\}$ could be identified with a submanifold of $S_3$. As $v_{2,2}$ is transverse to $S_3$ at $(0,0;k_0)$, we could ensure that the transversality holds everywhere on $S_3$ as in (1c).
	
\medskip

(2e). The flow generated by $v_{2,2}$ induces a local fibration at $(0,0;k_0)$, 
	\[\kappa_2:S_2\to S_3\times \R, \quad s_2\to (s_3,t_2).\]
 $\kappa_2$ satisfies $d\kappa_2(v_{2,2})=\partial/\partial t_2$ and $\kappa_2(s_3)=(s_3,0)$ for $s_3\in S_3$. The local fibration $\kappa_2$ preserves the parameters as well as in (1d). Let $\pi_2:S_2\to S_3$ be the composition of $\kappa_2$ and the projection $S_3\times \R\to S_3$. It is clear that $d\pi_2(v_{2,2})=0$.
	
	As in (1d), we also have that the function $(f_2-f_2\circ \pi_2)\circ \kappa_2^{-1}:S_3\times \R\to \R$ vanishes along $S_3\times \{0\}$ up to the first order. It is clear that $[(f_2-f_2\circ \pi_2)\circ \kappa_2^{-1}](s_3,0)=0$ for $(s_3,0)\in S_3\times \R$. In order to show that $\partial_{t_2}[(f_2-f_2\circ \pi_2)\circ \kappa_2^{-1}](s_3,0)=0$ for $(s_3,0)\in S_3\times \R$, we write
		\[\begin{split}
			&\quad\partial_{t_2}[(f_2-f_2\circ \pi_2)\circ \kappa_2^{-1}](s_3,t_2)\\
			&=\partial_{t_2}[f_2\circ \kappa_2^{-1}](s_3,t_2)\\
			&=[\mu_{2,2}\circ\kappa_2^{-1}](s_3,t_2)\cdot [g_{2,2}\circ\kappa_2^{-1}](s_3,t_2).
		\end{split}\]
	The first order derivative vanishes along $S_3\times \{0\}$ for the same reason in (1d). We can replace the function $[\mu_{2,2}\circ\kappa_2^{-1}](s_3,t_2)$ by the function $\mu_{2,2}(s_3)$ in the expression of the first order derivative as well as in (1d). We also get
		\[\partial_{t_2}^2[(f_2-f_2\circ \pi_2)\circ \kappa_2^{-1}](s_3,t_2)=\mu_{2,2}(s_3)\cdot [(v_{2,2}g_{2,2})\circ\kappa_2^{-1}](s_3,t_2).\]
	Furthermore, $[v_{2,2}g_{2,2}](0,0;k_0)$ is not 0 and $\mu_{2,2}(0,0;k_0)$ is not 0. Thus we can ensure that the second order derivative (in variable $t_2$) stays away from 0 in a neighborhood of $((0,0;k_0),0)$ in $S_3\times\R$.
	
\medskip	

(2f). We almost repeat the procedure in (1e). We have a local fibration at $(0,0;k_0)$, $\cka_2:S_2\to S_3\times \R, s_2\to (s_3,\ct_2)$ such that for $ (s_3,\ct_2)\in S_3\times \R$,
	\begin{eqnarray*}
		&&[(f_2-f_2\circ \pi_2)\circ \cka_2^{-1}](s_3,\ct_2)\\
		&&\quad\quad\quad\quad\quad\quad\quad\quad	=-\frac{1}{2}s\alpha_2(H_1+b_2H_2)\alpha_2(H'_1+c_2H'_2)\ct_2^2.
	\end{eqnarray*}
	To obtain the second local coordinate $\cx_2:S_1\to \R$, we take the (local) composition
		\[S_1\xrightarrow{\cka_1} S_2\times \R \xrightarrow{(\cka_2,\id)} S_3\times \R\times \R,\]
	which is also a local fibration at $(0,0;k_0)$, $S_1\to S_3\times \R^2$. Then $\cx_2:S_1\to \R$ will be the composition of $\pi_1:S_1\to S_2$ and $\ct_2:S_2\to \R$.
	
	The function $f_1:S_1\to \R$ could be written in sums
	\[\begin{split}
		f_1&=f_1\circ \pi_1+(f_1-f_1\circ \pi_1)\\
			&=f_2\circ \pi_1+(f_1-f_1\circ \pi_1)\\
			&=f_2\circ \pi_2\circ \pi_1+(f_2-f_2\circ\pi_2)\circ \pi_1+(f_1-f_1\circ \pi_1), 
	\end{split}\]
	here we used the fact that the function $f_1\circ \pi_1:S_1\to \R$ is the same as $f_2\circ\pi_1:S_1\to \R$. In view of the first and second local coordinates, we have
		\begin{eqnarray*}
		&&	f_1(b_2,c_2;k)=f_2\circ \pi_2\circ\pi_1(b_2,c_2;k)\\
		&&\quad\quad\quad\quad\quad\quad\quad\quad-\frac{1}{2}s\alpha_1(H)\alpha_1(H')\cx_1^2-\frac{1}{2}s\alpha_2(H)\alpha_2(H')\cx_2^2,
		\end{eqnarray*}
	here $H$ and $H'$ are put as in the statement of the theorem.
	
	The next task is to construct the third local coordinate $\cx_3:S_1\to \R$. Once we achieve that, we will almost be done. Before we list the manipulations and facts, we claim that $S_3$ is an open submanifold of $\R\times \R\times \rK_{H_1}s\rK_{H'_1}$. We observe that $S_3$ has dimension 3 which is the same as the dimension of $\R\times \R\times \rK_{H_1}s\rK_{H'_1}$. To prove this claim, it suffices to show that $g_{1,1}:S_1\to \R$ and $g_{2,2}:S_2\to \R$ both vanish on $\R\times \R\times \rK_{H_1}s\rK_{H'_1}$. We single out this fact as Lemma \ref{lem-v1} and present it after finishing the theorem.
	
	We list the manipulations and facts as well.

(3a). We showed that $v_{2,2}$ is transverse to $S_3$ everywhere in (2d). Thus we can take the projection of $v_{2,3}$ to $S_3$ with respect to $v_{2,2}$ and denote it by $v_{3,3}$. We want to show that $v_{3,3}$ coincides with the restriction of $v_{1,3}$ to $S_3$. First, $v_{1,3}$ is tangent to $S_3$ (or $\R\times \R\times \rK_{H_1}s\rK_{H'_1}$) since the left invariant vector field $\inv{X_3}{L}$ is tangent to $\rK_{H_1}s\rK_{H'_1}=s\rK_{H'_1}$. Thus the restriction of  $v_{2,3}$ (which is the projection of $v_{1,3}$) to $(\R\times \R \times \rK_{H_1}s\rK_{H'_1})\cap S_2$ coincides with the restriction of $v_{1,3}$ to $(\R\times \R \times \rK_{H_1}s\rK_{H'_1})\cap S_2$. Here, the restrictions make sense because  $(\R\times \R \times \rK_{H_1}s\rK_{H'_1})\cap S_2$ is identified with a submanifold of $S_2$. Again, $v_{3,3}$ which is the projection of $v_{2,3}$ coincides with the restriction of $v_{2,3}$ to $S_3$, thus it coincides with the restriction of $v_{1,3}$ to $S_3$.
	
	\medskip
	

(3b). Let $f_3:S_3\to \R$ be the restriction of the function $f_1:S_1\to \R$ (or $f_2:S_2\to \R$) to $S_3$. Consider the function $v_{3,3}f_3:S_3\to \R$. Here we can evaluate $v_{3,3}f_3$ directly without resorting to Lemma \ref{lem-id} and have
		\[[v_{3,3}f_3](b_2,c_2;k)=[v_{1,3}f_1](b_2,c_2;k),\quad (b_2,c_2;k)\in S_3\]
		because $v_{3,3}$ is the restriction of $v_{1,3}$ and $f_3:S_3\to \R$ is the restriction of $f_1:S_3\to \R$. Thus it is equal to
		\[-\alpha_3(H'_1+c_2H'_2)B(H_1+b_2H_2,\Ad k(X_{3,\alpha_3}-\theta X_{3,\alpha_3})).\]
		We look closer at the second factor,
		\[B(H_1,\Ad k(X_{3,\alpha_3}-\theta X_{3,\alpha_3}))+b_2 B(H_2,\Ad k(X_{3,\alpha_3}-\theta X_{3,\alpha_3})).\]
		We show that the first term in the sum is always 0 for $k\in \rK_{H_1}s$. $k\in \rK_{H_1}s$, so $\Ad k^{-1}H_1=s^{-1}H_1$. We have 
		\[B(H_1,\Ad k(X_{3,\alpha_3}-\theta X_{3,\alpha_3}))=B(s^{-1}H_1,X_{3,\alpha_3}-\theta X_{3,\alpha_3})=0.\]
		Thus the function $v_{3,3}f_3:S_3\to \R$ is a product of the following two functions
		\begin{equation} \label{eq:3b1}
		\mu_{3,3}(b_2,c_2;k):=-\alpha_3(H'_1+c_2H'_2)b_2=-b_2c_2\alpha_3(H'_2),
		\end{equation}
		\begin{equation} \label{eq:3b2}
			g_{3,3}(b_2,c_2;k):=B(H_2,\Ad k(X_{3,\alpha_3}-\theta X_{3,\alpha_3})).
		\end{equation}
		We point out that $\mu_{3,3}:S_3\to \R$ only depends on $b_2,c_2$ while $g_{3,3}:S_3\to \R$ only depends on $k$. Furthermore, $\mu_{3,3}:S_3\to \R$ carries the singularity $b_2\cdot c_2$.
		
		\medskip


(3c). We show that $g_{3,3}:S_3\to \R$ vanishes at $(b_2,c_2;k_0)\in S_3$ and $g_{3,3}$ is a submersion at $(0,0;k_0)$. According to the expression of $g_{3,3}$, it is clear that $g_{3,3}$ vanishes at $(b_2,c_2;k_0)$. Then we show $[v_{3,3}g_{3,3}](0,0;k_0)$ is not 0. Because we have the exact expressions of $g_{3,3}$ and $v_{3,3}$, we get 
		\[\begin{split}
			[v_{3,3}g_{3,3}](0,0;k_0)&=B([\Ad (k_0^{-1}) H_2,X_3],X_{3,\alpha_3}-\theta X_{3,\alpha_3})\\
			&=\epsilon \cdot s\alpha_3(H_2).
		\end{split}\]
		Since $H_2$ is regular, $s\alpha_3(H_2)$ is not 0.
		
			\medskip
		

(3d). As $g_{3,3}$ is a submersion at $(0,0;k_0)$, the level set $g_{3,3}^{-1}(\{0\})$ is locally a submanifold of $S_3$ and we denote it by $S_4$. In fact, $S_4$ is an open submanifold of $\R\times\R\times \{k_0\}$ because $g_{3,3}$ vanishes at $(\R\times\R\times \{k_0\})\cap S_3$ and $\R\times\R\times \{k_0\}$ has the same dimension 2 as $S_4$. Furthermore, it is obvious that $v_{3,3}$ is transverse to $S_4$ everywhere.
		
			\medskip
		

(3e). The flow generated by $v_{3,3}$ induces a local fibration at $(0,0;k_0)$, 
		\[\kappa_3:S_3\to S_4\times \R,\quad s_3\to (s_4,t_3).\] 
		Let $\pi_3:S_3\to S_4$ be the composition of $\kappa_3$ and the projection $S_4\times \R\to S_4$. Indeed, $\pi_3$ simply sends $(b_2,c_2;k)\in S_3$ to $(b_2,c_2;k_0)\in S_4$ due to exact expressions of $S_3$ and $S_4$. It is immediate to see that $(f_3-f_3\circ \pi_3)\circ \kappa_3^{-1}:S_4\times \R\to \R$ vanishes along $S_4\times \{0\}$ up to the first order, that is, for $(s_4,0)\in S_4\times \R$, 
		\[[(f_3-f_3\circ \pi_3)\circ \kappa_3^{-1}](s_4,0)=0,\] 
		\[\partial_{t_3}[(f_3-f_3\circ \pi_3)\circ \kappa_3^{-1}](s_4,0)=0 .\]
		Furthermore, we have
			\[\begin{split}
				&\quad\partial_{t_3}[(f_3-f_3\circ \pi_3)\circ \kappa_3^{-1}](s_4,t_3)\\
				&=[\mu_{3,3}\circ\kappa_3^{-1}](s_4,t_3)\cdot [g_{3,3}\circ\kappa_3^{-1}](s_4,t_3)\\
				&=\mu_{3,3}(s_4)\cdot [g_{3,3}\circ\kappa_3^{-1}](s_4,t_3).
			\end{split}\]
		Let us emphasize the significance of factoring out the term $[\mu_{3,3}\circ\kappa_3^{-1}](s_4,t_3)$ and replacing it with $\mu_{3,3}(s_4)$. They help us resolve the singularity of $b_2c_2$. 
		
		We look at the second order derivative in variable $t_3$,
		\begin{eqnarray*}
		&&\partial_{t_3}^2\left [(f_3-f_3\circ \pi_3)\circ \kappa_3^{-1}\right](s_4,t_3)\\
		&&\quad\quad\quad\quad\quad\quad\quad\quad=\mu_{3,3}(s_4)[(v_{3,3}g_{3,3})\circ \kappa_3^{-1}](s_4,t_3).
		\end{eqnarray*}
		We know that $[v_{3,3}g_{3,3}](b_2,c_2;k_0)$ is never 0. 
		
			\medskip
		

(3f). We apply the Taylor's formula to $(f_3-f_3\circ\pi_4)\circ \kappa_3^{-1}:S_4\times \R\to \R$ in the $t_3$ variable and get
		\begin{eqnarray*}
		&&[(f_3-f_3\circ\pi_3)\circ \kappa_3^{-1}](s_4,t_3) \\
		&&\quad \quad\quad\quad\quad\quad\quad\quad=\mu_{3,3}(s_4)\frac{t_3^2}{2}\int_0^1 2(1-z)[(v_{3,3}g_{3,3})\circ \kappa_3^{-1}](s_4,t_3z)dz.
		\end{eqnarray*}
		We point out that $\mu_{3,3}(s_4)=-b_2c_2\alpha_3(H'_2)$ with $\alpha_3(H'_2)\neq 0$. It is crucial to have both the factors $b_2,c_2$ outside the integral which permits us to later take the square root without involving singularities.
		
		We insert a constant $\epsilon\cdot s\alpha_3(H_2)\neq 0$ into the equation. The right side of the equation becomes
		\begin{eqnarray*}
		&&-\frac{1}{2}\epsilon\cdot b_2s\alpha_3(H_2)c_2\alpha_3(H'_2)\frac{t_3^2}{\epsilon\cdot s\alpha_3(H_2)}\\
		&& \quad \quad\quad\quad\quad\quad\quad\quad \cdot \int_0^1 2(1-z)[(v_{3,3}g_{3,3})\circ \kappa_3^{-1}](s_4,t_3z)dz.
		\end{eqnarray*}
		Beware that $b_2s\alpha_3(H_2)c_2\alpha_3(H'_2)=s\alpha_3(H_1+b_2H_2)\alpha_3(H'_1+c_2H'_2)$. Put
			\[\ct_3(s_4,t_3)=t_3\left[\frac{1}{\epsilon\cdot s\alpha_3(H_2)}\int_0^1 2(1-z)[(v_{3,3}g_{3,3})\circ \kappa_3^{-1}](s_4,t_3z)dz\right]^{1/2}.\]
		Then we get a local fibration at $((0,0;k_0),0)$, $\cka_3:S_3\to S_4\times \R, s_3\to (s_4,\ct_3)$ such that for $(s_4,t_3)\in S_4\times \R$,
		\[[(f_3-f_3\circ \pi_3)\circ \cka_3^{-1}](s_4,\ct_3)=-\frac{1}{2}\epsilon\cdot s\alpha_3(H)\alpha_3(H')\ct_3^2.\]
	To get the third local coordinate $\cx_3:S_1\to \R$ we shall take the (local) composition
	\[S_1\xrightarrow{\cka_1} S_2\times \R \xrightarrow{(\cka_2,\id)} S_3\times \R^2\xrightarrow{(\cka_3,\id^2)} S_4\times \R^3.\]
	
	With local coordinates $\cx_1,\cx_2,\cx_3:S_1\to \R$, we have that for $(b_2,c_2;k)\in S_1$,

	\[
		\begin{split}
			f_1(b_2,c_2;k)&=f_3\circ \pi_3\circ\pi_2\circ\pi_1(b_2,c_2;k)-\frac{1}{2}s\alpha_1(H)\alpha_1(H')\cx_1^2\\
			&-\frac{1}{2}s\alpha_2(H)\alpha_2(H')\cx_2^2-\frac{1}{2}\epsilon\cdot s\alpha_3(H)\alpha_3(H')\cx_3^2.
		\end{split}
	\]
	Note that $\pi_3\circ \pi_2\circ \pi_1(b_2,c_2;k)=(b_2,c_2,k_0)$. So the expression above is exactly the equation in the statement.
	
	Finally, we conclude our proof by adding the last two coordinates $(b_2,c_2;k)\to b_2$ and  $(b_2,c_2;k)\to c_2$. We shall look at the map $S_4\to \R^2, (b_2,c_2;k_0)\to (b_2,c_2)$. The projection is a local diffeomorphism at $(0,0;k_0)$ because $S_4$ is an open submanifold of $\R\times \R\times \{k_0\}$. Then the (local) composition
	\[S_1\rightarrow \dots \rightarrow S_4\times \R^3\to \R^2\times \R^3\]
	gives the local coordinate system we seek after arranging the components in the suitable order. Furthermore, it is clear that $(b_2,c_2;k_0)$ is always mapped to $(b_2,c_2;0,0,0)$.
\end{proof}

\begin{lemma} \label{lem-v1}
Let $H\in \fa$ be singular and not 0 and let $\alpha$ be the unique positive root that vanishes at $H$. Then we have
	\[B(H_1,\Ad k(X))=0,\quad H_1\in\fa, k\in \rK_H, X\in \fq\ominus \fs_\alpha\]
\end{lemma}

\begin{proof}
$B(H_1,\Ad k(X))=B(\Ad{k^{-1}}H_1,X)$. For $k\in \rK_H\subseteq \rG_H$ and $H_1\in \fa\subseteq \cent{\fg}{H}=\Lie(\rG_H)$, $\Ad{k^{-1}}H_1\in \cent{\fg}{H}\cap \fs$. We know $\cent{\fg}{H}=\fa\oplus\fg_{\alpha}\oplus \fg_{-\alpha}$. So $\Ad{k^{-1}}H_1\in\fa\oplus \fs_\alpha$. While $X\in \fq\ominus \fs_\alpha$, $B(\Ad{k^{-1}}H_1,X)$ is 0.
\end{proof}

Now we move to the local normal forms of $f_1$ for $k_0\in \rK_{H_1}s\rK_{H'_1}\backslash \rM'$.

\begin{corollary}\label{cor-n4}
Assume $k_0$ is contained in the critical manifold $\rK_{H_1}s\rK_{H'_1}$ which has dimension 1. Let $\alpha_1,\alpha_2,\alpha_3\in \Rr^+$ be arranged as in Theorem \ref{thm-n1}. 

If $k_0\in \rK_{H_1}s\rK_{H'_1}\backslash \rM'$, then there exists a local coordinate system at $(0,0;k_0)$, $\kappa:S_1\to \R^5$ which preserves the parameters, that is, $\kappa(b_2,c_2;k)=(b_2,c_2;\cdot)$. $\kappa$ satisfies
	\[\begin{split}
	f_1\circ \kappa^{-1}(b_2,c_2;\cx_1,\cx_2,\cx_3)&=f_1(b_2,c_2;k_0)-\frac{1}{2}s\alpha_1(H)\alpha_1(H')\cx_1^2\\
	&-\frac{1}{2}s\alpha_2(H)\alpha_2(H')\cx_2^2
	-s\alpha_3(H)\alpha_3(H')\cx_3,
	\end{split}\]
	here $H=H_1+b_2H_2,H'+H'_1+c_2H'_2$. Besides, $\kappa$ always maps $(b_2,c_2;k_0)$ to $(b_2,c_2;0,0,0)$.
\end{corollary}

\begin{proof}
As $k_0\in \rK_{H'_1}s\rK_{H'_1}=s\rK_{H'_1}$, $k_0$ could be written as $k_sk_1$ with $k_s$ a representative of $s$ and $k_1\in \rK_{H'_1}$. According to our assumption $k_0\notin \rM'$, $k_1\in \rK_{H'_1}\backslash \rM'$.

We almost repeat what we have done in the proof of Theorem \ref{thm-n1} line by line. We only modify the argument at two or three places. In the following discussions, the labels like (1a) refer to those in the proof of Theorem \ref{thm-n1}.

Let the orthonormal basis $\{X_1,X_2,X_3\}$ of $\fk$ and vertical vector fields  $v_{1,1},v_{1,2},v_{1,3}$ on $S_1$ be chosen as in the proof of Theorem \ref{thm-n1}. 

We construct the first local coordinate $\cx_{1}:S_1\to \R$. As in (1a), we consider the function $v_{1,1}f_1:S_1\to \R$ and we have $v_{1,1}f_1=\mu_{1,1}g_{1,1}$, here $\mu_{1,1},g_{1,1}:S_1\to \R$ have exactly the same expressions \eqref{eq:1a1} and \eqref{eq:1a2} in (1a). 

As in (1b), we show that $g_{1,1}:S_1\to \R$ vanishes at $(b_2,c_2;k_0)\in S_1$ and $g_{1,1}$ is a submersion at $(0,0;k_0)$. This step is slightly different.  We have
	\[g_{1,1}(b_2,c_2;k_0)=B(s^{-1}(H_1+b_2H_2),\Ad k_1(X_{1,\alpha_1}-\theta X_{1,\alpha_1})).\] 
	As $k_1\in \rK_{H'_1},X_{1,\alpha_1}-\theta X_{1,\alpha_1}\in \fq\ominus \fs_{\alpha_3}, s^{-1}(H_1+b_2H_2)\in \fa$, Lemma \ref{lem-v1} tells that $g_{1,1}(b_2,c_2;k_0)$ is 0. We verify that $g_{1,1}$ is a submersion at $(0,0;k_0)$ in the same way as in (1b).  We have 

\[\Ad k_0^{-1} H_1=\Ad (k_1^{-1})s^{-1}H_1=s^{-1}H_1\]
because $k_1\in \rK_{H'_1}=\rK_{s^{-1}H_1}$. So $[v_{1,1}g_{1,1}](0,0;k_0)=s\alpha_1(H_1)\neq 0$.

	 The submanifold $S_1$ is defined as in (1c). We exactly redo (1d). We get a local fibration at $(0,0;k_0)$, $\kappa_1:S_1\to S_2\times \R$ and $\pi_1:S_1\to S_2$. We have $(f_1-f_1\circ \pi_1)\circ \kappa_1^{-1}:S_2\times \R\to \R$ vanishes along $S_2\times \{0\}$ up to the first order. We also have that for $(s_2,t_1)\in S_2\times \R$,
	 	\begin{eqnarray*}
		&&	\partial_{t_1}^2\left[(f_1-f_1\circ \pi_1)\circ \kappa_1^{-1}\right](s_2,t_1)\\
		&&\quad\quad\quad\quad\quad\quad\quad\quad=[\mu_{1,1}\circ\kappa_1^{-1}](s_2,t_1)\cdot [ v_{1,1}g_{1,1}\circ \kappa_1^{-1}](s_2,t_1)
		\end{eqnarray*}
		which has value $-\alpha_1(H'_1)s\alpha_1(H_1)\neq 0$ at $((0,0;k_0),0)$. 
		
	By repeating (1e), we get the first local coordinate $\cx_1:S_1\to \R$. We want to point out that the objects we encoutered here like $\kappa_1,\cx_1$ are indeed the same as those in the proof of Theorem \ref{thm-n1}. However, their exsitence are valid in different neighborhoods in $S_1$ because the term `local' makes sense in the neighborhoods of the different $(0,0;k_0)$.
	
	As we see, the steps for constructing the first local coordinate are almost the same as in the proof of Theorem \ref{thm-n1}. We claim the steps for constructing the second coordinate are almost the same as for the first coordinate. The readers could run through them one by one. So we skip this part and move to the construction of the third local coordinate while assuming objects like $f_2,S_3,\pi_3,\cx_2$ have been taken.
	
	We also have that $S_3$ is an open submanifold of $\R\times \R\times \rK_{H_1}s\rK_{H'_1}$. We repeat (3a) and get $v_{3,3}$ which coincides with the restriction of $v_{1,3}$ to $S_3$. We have $v_{3,3}f_{3}=\mu_{3,3}g_{3,3}$, here $\mu_{3,3},g_{3,3}:S_3\to \R$ have the same expressions \eqref{eq:3b1} and \eqref{eq:3b2} in (3b).
	
	Now we have a different situation which simplifies our argument. We claim that $g_{3,3}(0,0;k_0)$ is not 0. 
	\[g_{3,3}(0,0;k_0)=B(s^{-1}H_2,\Ad k_1 (X_{3,\alpha_3}-\theta X_{3,\alpha_3})).\]
	We single out this fact as Lemma \ref{lem-v2} and present it after this proof.
	
	We set $S_4$ to be $\{(b_2,c_2;k_0)\in S_3\}$, which is an open submanifold of $\R\times \R\times \{k_0\}$. The flow generated by $v_{3,3}$ induces a local fibration at $(0,0;k_0)$, $\kappa_3:S_3\to S_4\times \R$. $\pi_3:S_3\to S_4$ exactly sends $(b_2,c_2;k)\in S_3$ to $(b_2,c_2;k_0)\in S_4$. We apply the Taylor's formula to $(f_3-f_3\circ \pi_3)\circ \kappa_3^{-1}:S_4\times \R\to \R$ in the $t_3$ variable,
	\begin{eqnarray*}
		&&\left [(f_3-f_3\circ \pi_4)\circ \kappa_3^{-1}\right](s_4,t_3)\\
		&&\quad\quad\quad\quad\quad\quad =-c_2\alpha_3(H'_2)b_2t_3\int_0^1 \left [g_{3,3}\circ\kappa_3^{-1}\right](s_4,t_3z)dz.
	\end{eqnarray*}
	We put
	\[\ct_3(s_4,t_3)=\frac{t_3}{s\alpha_3(H_2)}\int_0^1 [g_{3,3}\circ\kappa_3^{-1}](s_4,t_3z)dz.\] 
	We get a local fibration at $(0,0;k_0)$, $\cka_3:S_3\to S_4\times \R, s_3\to (s_4,\ct_3)$ such that
	\[\left [(f_3-f_3\circ \pi_4)\circ \cka_3^{-1}\right](s_4,\ct_3)=-s\alpha_3(H)\alpha_3(H')\ct_3.\]
	Going through the rest of the routine, we establish the corollary.
\end{proof}

\begin{lemma} \label{lem-v2}
Let $H\in \fa$ be singular and not 0, let $\alpha$ be the unique positive root that vanishes at $H$ and let $X\in \fs_\alpha$ not 0. Then for $k\in \rK_H$, $B(\alpha^\vee,\Ad k(X))$ is 0 if and only if $k\in \rM'_H$.
\end{lemma}

\begin{proof}
If $k\in \rM'_H$, it is obvious $B(\alpha^\vee,\Ad k(X))=0$ for $\Ad k^{-1}\alpha^\vee\in\fa$ and $X\in \fq$. 

We show the reverse direction. Let $\fg_1$ be the Lie subalgebra $\R\alpha^\vee\oplus \fg_\alpha\oplus \fg_{-\alpha}$ which is an ideal of $\cent{\fg}{H}$. Let $\rG_1$ be the analytic subgroup of $\rG$ with $\Lie(\rG_1)=\fg_1$. 

As $\Lie(\rK_H)=\fk_\alpha$ is a Lie subalgebra of $\fg_1$, $\rK_H^0$ is an analytic subgroup of $\rG_1$. As $k\in \rK_H=\rM\rK_H^0$, it could be written as $mk_1$, here $m\in \rM$ and $k_1\in \rK_H^0$. Beware that this decomposition is not unique. Then \[B(\alpha^\vee,\Ad k(X))=B(\Ad k^{-1} \alpha^\vee,X)=B(\Ad k_1^{-1} \alpha^\vee,X).\]

 As $\Ad k_1^{-1}\alpha^\vee\in \fg_1\cap \fs=\R\alpha^\vee\oplus \fs_\alpha$ and $\fs_\alpha=\R X$, $B(\Ad k_{1}^{-1}\alpha^\vee,X)=0$ implies that $\Ad k_1^{-1}\alpha^\vee$ lands in $\R\alpha^{\vee}$. Since $\len{\Ad k_1^{-1}\alpha^\vee}=\len{\alpha^\vee}$, we have $\Ad k_1^{-1}\alpha^\vee=\pm \alpha^\vee$. Note that $\Ad k_1^{-1}H=H$ and $\fa=\R H\oplus \R\alpha^\vee$. Thus $\Ad k_1^{-1}\fa=\fa$, that is, $k_1^{-1}\in \rM'$. As $k_1\in \rK_{H}^0$, $k=mk_1\in \rM'_H$.
\end{proof}

Now we discuss the situation that $k_0$ is contained in the critical manifold $\rK_{H_1}s\rK_{H'_1}$ with dimension 2.

The critical manifold $\rK_{H_1}s\rK_{H'_1}=s\rK_{s^{-1}H_1}\rK_{H'_1}$ has dimension 2 if and only if $\rK_{s^{-1}H_1}\neq \rK_{H'_1}$. It is equivalent to $s^{-1}H_1\neq \pm H'_1$, namely, $s^{-1}H_1$ and $H'_1$ belong to different root hyperplanes.

We consider the following partition of the critical manifold $\rK_{H_1}s\rK_{H'_1}$,
\[\begin{split}\rK_{H_1}s\rK_{H'_1}&=\left(\rK_{H_1}s\rK_{H'_1}\cap \rM'\right)\cup \left(\rK_{H_1}s\rK_{H'_1}\backslash (\rM'_{H_1}s\rK_{H'_1}\cup \rK_{H_1}s\rM'_{H'_1})\right)\\
&\cup \left(\rK_{H_1}s\rM'_{H'_1}\backslash \rM'_{H_1}s\rK_{H'_1}\right)\cup \left(\rM'_{H_1}s\rK_{H'_1}\backslash \rK_{H_1}s\rM'_{H'_1}\right).
\end{split}\]
This partition makes sense for the following reason. First, we have the fact that \[\rM'_{H_1}s\rK_{H'_1}\cap \rK_{H_1}s\rM'_{H'_1}=\rM'_{H_1}s\rM'_{H'_1}\subseteq \rM'.\]
 Then we show that $\rK_{H_1}s\rK_{H'_1}\cap \rM' \subseteq \rM'_{H_1}s\rM'_{H'_1}$, so \[\rK_{H_1}s\rK_{H'_1}\cap \rM'=\rM'_{H_1}s\rK_{H'_1}\cap \rK_{H_1}s\rM'_{H'_1}. \]
 Assuming that $k_0=k_1k_sk_2\in \rK_{H_1}s\rK_{H'_1}\cap \rM'$, we show $k_1\in \rM'_{H_1}$ and $k_2\in \rM'_{H'_1}$. Because $k_0k_2^{-1}=k_1k_s$, we have $\Ad (k_0k_2^{-1})H'_1=\Ad (k_0)H'_1\in\fa$ and $\Ad(k_1k_s)s^{-1}H_1=\Ad(k_1)H_1=H_1$. Since $\fa=\R H'_1+\R s^{-1}H_1$, $k_0k_2^{-1}=k_1k_s\in \rM'$, thus $k_1,k_2\in \rM'$. So $k_1\in \rM'_{H_1},k_2\in \rM'_{H'_1}$.

We first work on the local normal form of $f_1$ for $k_0\in \rK_{H_1}s\rK_{H'_1}\cap \rM'$, then for $k_0\in \rK_{H_1}s\rK_{H'_1}\backslash(\rM'_{H_1}s\rK_{H'_1}\cup \rK_{H_1}s\rM'_{H'_1})$, and lastly for the intermediate cases $k_0\in \rK_{H_1}s\rM'_{H'_1}\backslash \rM'_{H_1}s\rK_{H'_1}$ and $k_0\in \rM'_{H_1}s\rK_{H'_1}\backslash \rK_{H_1}s\rM'_{H'_1}$.

We arrange the positive roots $\alpha_1,\alpha_2,\alpha_3$ in the way such that $\alpha_2$, resp. $\alpha_3$ is the unique positive root which vanishes at $H'_1$, resp. $s^{-1}H_1$. Then we have $\Lie(\rK_{H'_1})=\fk_{\alpha_2}$, $\Lie(\rK_{s^{-1}H_1})=\fk_{\alpha_3}$ and $\Lie(\rK_{H_1})=\Ad (k_s) \fk_{\alpha_3}$. 

We fix a representative $k_s\in \rM'$ of $s\in \rW$. As $\rK_{H_1}s\rK_{H'_1}=\rK_{H_1}k_s\rK_{H'_1}$, $k_0$ could be written as $k_1k_sk_2$, with $k_1\in \rK_{H_1},k_2\in \rK_{H'_1}$. Beware that this decomposition is not unique.

We identify $\tg{\rK_{H_1}s\rK_{H'_1}}{k_0}$ with a subspace of $\tg{\rK}{k_0}=\fk$ and have
	\[\tg{\rK_{H_1}s\rK_{H'_1}}{k_0}=\fk_{\alpha_2}+\Ad(k_0^{-1})\Ad(k_s)\fk_{\alpha_3}.\] 
As $\Ad(k_0^{-1})\Ad(k_s)\fk_{\alpha_3}=\Ad (k_2^{-1}k_s^{-1}k_1^{-1}k_s)\fk_{\alpha_3}$ and $k_s^{-1}k_1^{-1}k_s\in \rK_{s^{-1}H_1}$, we have $\Ad(k_0^{-1}k_s)\fk_{\alpha_3}=\Ad (k_2^{-1})\fk_{\alpha_3}$. Since $k_2\in \rK_{H'_1}$ and $\dim \Lie(\rK_{H'_1})=\dim \fk_{\alpha_2}=1$, we have $\fk_{\alpha_2}=\Ad (k_2^{-1})\fk_{\alpha_2}$. Hence the sum
	\[\tg{\rK_{H_1}s\rK_{H'_1}}{k_0}=\Ad(k_2^{-1})\fk_{\alpha_2}+\Ad(k_2^{-1})\fk_{\alpha_3}\]
	is indeed an orthogonal sum with respect to $B_\theta$. Furthermore, the orthogonal complement of $\tg{\rK_{H_1}s\rK_{H'_1}}{k_0}$ in $\fk$ is $\Ad (k_2^{-1})\fk_{\alpha_1}$.

\begin{theorem} \label{thm-n2}
Assume $k_0$ is contained in the critical manifold $\rK_{H_1}s\rK_{H'_1}$ which has dimension 2. Let $\alpha_1,\alpha_2,\alpha_3$ be arranged as above, that is, $\alpha_2(H'_1)$ and $s\alpha_3(H_1)$ are 0. 

If $k_0\in \rK_{H_1}s\rK_{H'_1}\cap \rM'$, then there exists a local coordinate system at $(0,0;k_0)$, $\kappa:S_1\to \R^5$ which preserves the parameters. $\kappa$ satisfies
	\[
		\begin{split}
			f_1\circ \kappa^{-1}(b_2, c_2; \cx_1, \cx_2, \cx_3)&=f_1(b_2, c_2; k_0)-\frac{1}{2}s\alpha_1(H)\alpha_1(H')\cx_1^2\\
			&-\frac{1}{2}\epsilon_2 s\alpha_2(H)\alpha_2(H')\cx_2^2-\frac{1}{2}\epsilon_3 s\alpha_3(H)\alpha_3(H')\cx_3^2,
		\end{split}
	\]
	here $H=H_1+b_2H_2,H'=H'_1+c_2H'_2$, $\epsilon_2,\epsilon_3$ are constants $\pm 1$  which are determined explicitly. Besides, $\kappa$ always maps $(b_2,c_2;k_0)$ to $(b_2,c_2;0,0,0)$. 
\end{theorem}

\begin{proof}
Due to the arrangement of the positive roots, we have that $\alpha_1(H'_1)$, $\alpha_3(H'_1)$, $s\alpha_1(H_1)$, $s\alpha_2(H_1)$ are all not 0.

According to the assumption of $k_0=k_1k_sk_2\in \rK_{H_1}s\rK_{H'_1}\cap \rM'=\rM'_{H_1} s \rM'_{H_1'}$, we have $k_1\in \rM'_{H_1}$ and $k_2\in \rM'_{H'_1}$. 

As $k_2\in \rM'_{H'_1}$, we have $\Ad(k_2)\alpha_2^\vee=\pm \alpha_2^\vee$ then set $\epsilon_2=\Ad(k_2)\alpha_2^\vee/\alpha_2^\vee$. As $k_1\in \rM'_{H_1}$ and $k_s^{-1}k_1^{-1}k_s\in \rM'_{s^{-1}H_1}$, we  have $\Ad (k_s^{-1}k_1^{-1}k_s)\alpha_3^\vee=\pm \alpha_3^\vee$ then set $\epsilon_3=\Ad (k_s^{-1}k_1^{-1}k_s)\alpha_3^\vee/\alpha_3^\vee$.

Let $\{X_1,X_2,X_3\}$ be an orthonormal basis of $\fk$ such that $X_i\in \fk_{\alpha_i}$, for $i=1,2,3$. Let $v_1$ be the left invariant vector field $\inv{(\Ad(k_2^{-1})X_1)}{L}$, let $v_2$ be the left invariant vector field $\inv{X_2}{L}$ and let $v_3$ be the right invariant vector field $\inv{(\Ad(k_s)X_3)}{R}$. We point out that the restrictions of $v_2$ and $v_3$ to the critical manifold $\rK_{H_1}s\rK_{H_1'}$ are tangent to $\rK_{H_1}s\rK_{H_1'}$. We also have $v_1$ transverse to $\rK_{H_1}s\rK_{H_1'}$ at $k_0$. Let $v_{1,1},v_{1,2},v_{1,3}$ be the canonical extensions of $v_1,v_2,v_3$ to $S_1=\R\times\R\times \rK$.

We construct the first local coordinate $\cx_1:S_1\to \R$. Consider the function $v_{1,1}f_1:S_1\to \R$,
	\begin{eqnarray*}
	&&[v_{1,1}f_1](b_2,c_2;k)=-\alpha_1(H'_1+c_2\Ad(k_2)H'_2)\\
	&&\quuad\quad\quad \cdot B(H_1+b_2H_2,\Ad (kk_2^{-1})(X_{1,\alpha_1}-\theta X_{1,\alpha_1})).
	\end{eqnarray*}
	We put $\mu_{1,1},g_{1,1}:S_1\to \R$,
		\[\mu_{1,1}(b_2,c_2;k):=-\alpha_1(H'_1+c_2\Ad(k_2)H'_2),\]
		\[g_{1,1}(b_2,c_2;k):=B(H_1+b_2H_2,\Ad (kk_2^{-1})(X_{1,\alpha_1}-\theta X_{1,\alpha_1})).\]
	Note that $\Ad(k_2)H'_2=\epsilon_2 H'_2$.
	We show that $g_{1,1}:S_1\to \R$ vanishes at $\{(b_2,c_2,k'k_0)\in S_1|k'\in \rK_{H_1}\}$. We have
		\[\begin{split}
			g_{1,1}(b_2,c_2;k'k_0)&=B(H,\Ad (k'k_1k_s)(X_{1,\alpha_1}-\theta X_{1,\alpha_1}))\\
			                                    &=B(s^{-1}H, \Ad (k_s^{-1}k'k_1k_s)(X_{1,\alpha_1}-\theta X_{1,\alpha_1})),
		\end{split}\]
		here $H$ stands for $H_1+b_2H_2$ for convenience. 
		
		As $k_s^{-1}k'k_1k_s\in \rK_{s^{-1}H_1}$, $\alpha_3(s^{-1}H_1)=0$ and $X_{1,\alpha_1}-\theta X_{1,\alpha_1}\in \fq\ominus \fs_{\alpha_3}$, Lemma \ref{lem-v1} tells that $g_{1,1}(b_2,c_2;k'k_0)$ is 0 when $k'\in \rK_{H_1}$.
		
		We show that $g_{1,1}$ is a submersion at $(0,0;k_0)$ by showing $[v_{1,1}g_{1,1}](0,0;k_0)$ not 0.
			\[\begin{split}
			[v_{1,1}g_{1,1}](0,0;k_0)&=B(H_1,\Ad k_0[\Ad k_2^{-1}X_1,\Ad k_2^{-1}(X_{1,\alpha_1}-\theta X_{1,\alpha_1})])\\
			&=B(s^{-1}H_1,[X_1,X_{1,\alpha_1}-\theta X_{1,\alpha_1}])\\
			&=s\alpha_1(H_1)\neq 0.
			\end{split}\]
		As in the proof of Theorem \ref{thm-n1}, let $S_2$ be the submanifold of $S_1$ corresponding to the level set $g_{1,1}^{-1}(\{0\})$. Then $\{(b_2,c_2;k'k_0)\in S_2|k'\in \rK_{H_1}\}$ could be identified with a 3-dimension submanifold of $S_2$. We introduce a local fibration at $(0,0;k_0)$, $\kappa_1:S_1\to S_2\times \R$ and set $\pi_1:S_1\to S_2$ accordingly. Applying the same method in the proof of Theorem \ref{thm-n1}, we obtain a local fibration at $(0,0;k_0)$, $\cka_1:S_1\to S_2\times \R, s_1\to (s_2,\ct_1)$ such that
		\[f_1(b_2,c_2;k)=f_1\circ \pi_1(b_2,c_2;k)-\frac{1}{2}s\alpha_1(H)\alpha_1(H')\ct_1^2.\]
		
		We construct the second local coordinate $\cx_2:S_1\to \R$. Let $v_{2,2},v_{2,3}$ be the projections of $v_{1,2},v_{1,3}$ to $S_2$ with respect to $v_{1,1}$. Consider the functions $v_{1,2}f_1,v_{1,3}f_1:S_1\to \R$. We have for $(b_2,c_2;k)\in S_1$,
		\[[v_{1,2}f_1](b_2,c_2;k)=\mu_{1,2}(b_2,c_2;k)g_{1,2}(b_2,c_2;k),\]
		\begin{equation}\label{eq:4-1}
			\mu_{1,2}(b_2,c_2;k):=-\alpha_2(H'_1+c_2H'_2)=-c_2\alpha_2(H'_2),
		\end{equation}
		\begin{equation}\label{eq:4-2}
			g_{1,2}(b_2,c_2;k):=B(H_1+b_2H_2,\Ad k(X_{2,\alpha_2}-\theta X_{2,\alpha_2})),
		\end{equation}
			\[[v_{1,3}f_1](b_2,c_2;k)=\mu_{1,3}(b_2,c_2;k)g_{1,3}(b_2,c_2;k),\]
		\begin{equation}\label{eq:4-3}
		\mu_{1,3}(b_2,c_2;k):=s\alpha_3(H_1+b_2H_2)=b_2s\alpha_3(H_2),
		\end{equation}
		\begin{equation}\label{eq:4-4}
		g_{1,3}(b_2,c_2;k):=B(\Ad k_s(X_{3,\alpha_3}-\theta X_{3,\alpha_3}),\Ad k(H'_1+c_2H'_2)),
		\end{equation}
	here $\alpha_2(H'_1)=s\alpha_3(H_1)=0$ according to our choice of $\alpha_2,\alpha_3$.
	
	Let $f_2:S_2\to \R$ be the restriction of $f_1:S_1\to \R$ to $S_2$. We also restrict $\mu_{1,2},\mu_{1,3},g_{1,2},g_{1,3}:S_1\to \R$ to $S_2$ and denote these restrictions by $\mu_{2,2},\mu_{2,3},g_{2,2},g_{2,3}:S_2\to \R$. Since $v_{1,1}f_1:S_1\to \R$ vanishes on $S_2$, Lemma \ref{lem-id} gives that for $s_2\in S_2$,
		\[[v_{2,2}f_2](s_2)=[v_{1,2}f_1](s_2),\quad [v_{2,3}f_2](s_2)=[v_{1,3}f_1](s_2).\]
We also have for $s_2\in S_2$,
	\[[v_{2,2}f_2](s_2)=\mu_{2,2}(s_2)g_{2,2}(s_2),\quad [v_{2,3}f_2](s_2)=\mu_{2,3}(s_2)g_{2,3}(s_2).\]
	
	We show that $g_{2,2}:S_2\to \R$ vanishes at $\{(b_2,c_2;k'k_0)\in S_2|k'\in\rK_{H_1}\}$.
		\[\begin{split}
			g_{2,2}(b_2,c_2;k'k_0)&=B(H,\Ad (k'k_0)(X_{2,\alpha_2}-\theta X_{2,\alpha_2}))\\
										&=B(s^{-1}H,\Ad (k_s^{-1}k'k_1k_s)\Ad (k_2)(X_{2,\alpha_2}-\theta X_{2,\alpha_2})).
		\end{split}\]
	As $k_2\in \rM'_{H'_1}$, $X_{2,\alpha_2}-\theta X_{2,\alpha_2}\in \fs_{\alpha_2}$ and $\alpha_2(H'_1)=0$, we have 
		\[\Ad (k_2)(X_{2,\alpha_2}-\theta X_{2,\alpha_2})=\pm (X_{2,\alpha_2}-\theta X_{2,\alpha_2})\in \fs_{\alpha_2}.\] 
	Since $k_s^{-1}k'k_1k_s\in \rK_{s^{-1}H_1}$, $\alpha_3(s^{-1}H_1)=0$ and $\fs_{\alpha_2}\subseteq \fq\ominus \fs_{\alpha_3}$, Lemma \ref{lem-v1} tells that $g_{2,2}(b_2,c_2;k'k_0)=0$ when $(b_2,c_2;k'k_0)\in S_2$. 
	
	We show that $g_{2,2}:S_2\to \R$ is a submersion at $(0,0;k_0)$ by showing $[v_{2,2}g_{2,2}](0,0;k_0)\neq 0$. We first show that $[v_{1,1}g_{1,2}](0,0;k_0)$ is 0. We have
	\[\begin{split}
		[v_{1,1}g_{1,2}](0,0;k_0)&=B(H_1,\Ad k_0[\Ad k_2^{-1} X_1,X_{2,\alpha_2}-\theta X_{2,\alpha_2}])\\
	     &=B(s^{-1}H_1, [X_1,\Ad k_2(X_{2,\alpha_2}-\theta X_{2,\alpha_2})])\\
	     &=\alpha_1(s^{-1}H_1)B(X_{1,\alpha_1}-\theta X_{1,\alpha_1},\Ad k_2(X_{2,\alpha_2}-\theta X_{2,\alpha_2})).
	\end{split}\]
	Since $X_{1,\alpha_1}-\theta X_{1,\alpha_1}\in \fs_{\alpha_1}$ and $\Ad k_2(X_{2,\alpha_2}-\theta X_{2,\alpha_2})\in \fs_{\alpha_2}$, we know that $[v_{1,1}g_{1,2}](0,0;k_0)$ is 0.
	
	Then we show that $[v_{1,2}g_{1,2}](0,0;k_0)$ is not 0. We have
	\[\begin{split}
		[v_{1,2}g_{1,2}](0,0;k_0)&=B(H_1,\Ad k_0[X_2,X_{2,\alpha_2}-\theta X_{2,\alpha_2}])\\
										&=B(s^{-1}H_1,\Ad k_2[X_2,X_{2,\alpha_2}-\theta X_{2,\alpha_2}])\\
										&=B([\Ad k_2^{-1}s^{-1}H_1,X_2],X_{2,\alpha_2}-\theta X_{2,\alpha_2})\\
										&=\alpha_2(\Ad k_2^{-1}s^{-1}H_1)\\
										&=B(\Ad k_2 (\alpha_2^\vee),s^{-1}H_1)\\
										&=B(\epsilon_2 \alpha_2^\vee,s^{-1}H_1)\\
										&=\epsilon_2 s\alpha_2(H_1).
	\end{split}\]
	According to our choice of $\alpha_2$, $s\alpha_2(H_1)$ is not 0. Then Lemma \ref{lem-id} tells that $[v_{2,2}g_{2,2}](0,0;k_0)=[v_{1,2}g_{1,2}](0,0;k_0)\neq 0$.
	
	Let $S_3$ be the submanifold of $S_2$ through $(0,0;k_0)$ corresponding to the level set $g_{2,2}^{-1}(\{0\})$. As $g_{2,2}$ vanishes at $\{(b_2,c_2;k'k_0)\in S_2|k'\in \rK_{H_1}\}$, $S_3$ is an open submanifold of $\R\times \R\times \rK_{H_1}k_0$.
	
	 The flow generated by $v_{2,2}$ induces a local fibration at $(0,0;k_0)$, $\kappa_2:S_2\to S_3\times \R, s_2\to (s_3,t_2)$ and we set $\pi_2:S_2\to S_3$ accordingly. We apply the Taylor's formula to $(f_2-f_2\circ \pi_2)\circ \kappa_2^{-1}:S_3\times \R\to \R$, then have
	 \[\begin{split}
	 	&\quad [(f_2-f_2\circ\pi_2)\circ \kappa_2^{-1}](s_3,t_2)\\
	 	&=\frac{1}{2}t_2^2\int_0^1 2(1-z)\partial_{t_2}^2[(f_2-f_2\circ\pi_2)\circ \kappa_2^{-1}](s_3,t_2z)dz\\
	 	&=\frac{1}{2}\mu_{2,2}(s_3)t_2^2\int_0^1 2(1-z)[(v_{2,2}g_{2,2})\circ \kappa_2^{-1}](s_3,t_2z)dz.
	\end{split}\]
	We insert $\epsilon_2s\alpha_2(H_1+b_2H_2)$ into the right side of the equation, then introduce a local diffeomorphism at $((0,0;k_0),0)$, $S_3\times \R\to S_3\times \R, (s_3,t_2)\to (s_3,\ct_2)$ by setting
	\[\ct_2(s_3,t_2)=t_2\Big[\frac{1}{\epsilon_2 s\alpha_2(H_1+b_2 H_2)}\int_0^1 2(1-z)[(v_{2,2}g_{2,2})\circ \kappa_2^{-1}](s_3,t_2z)dz \Big]^{1/2}.\]
	Then we obtain the second local coordinate $\cx_2:S_1\to \R$ such that
	\begin{eqnarray*}
	&&f_1(b_2,c_2;k)=f_2\circ \pi_2\circ\pi_1(b_2,c_2;k)\\
	&&\quuad\quad\quad -\frac{1}{2}s\alpha_1(H)\alpha_1(H')\cx_1^2-\frac{1}{2}\epsilon_2 s\alpha_2(H)\alpha_2(H')\cx_2^2.
	\end{eqnarray*}	
	
	We construct the third local coordinate $\cx_3:S_1\to \R$. Let $v_{3,3}$ be the projection of $v_{2,3}$ to $S_3$ with respect to $v_{2,2}$. As we have pointed out that $S_3$ is an open submanifold of $\R\times \R\times \rK_{H_1}k_0$, we claim that $v_{3,3}$ coincides with the restriction of $v_{1,3}$ to $S_3$. It suffices to show that $v_3$ is tangent to $\rK_{H_1}k_0$ which is true since $v_3=\inv{(\Ad(k_s)X_3)}{R}$ and $\Ad(k_s)X_3\in \Lie(\rK_{H_1})$.
	
	Let $f_3:S_3\to \R$ be the restriction of the function $f_1:S_1\to \R$ (or $f_2:S_2\to \R$) to $S_3$. Then we have
		\[[v_{3,3}f_3](b_2,c_2;k)=[v_{1,3}f_1](b_2,c_2;k),\quad (b_2,c_2;k)\in S_3.\]
	Let $\mu_{3,3},g_{3,3}:S_3\to \R$ be the restrictions of $\mu_{1,3},g_{1,3}:S_1\to \R$ to $S_3$. We have $g_{3,3}(b_2,c_2;k_0)=0$ for
		\[g_{3,3}(b_2,c_2;k_0)=B(X_{3,\alpha_3}-\theta X_{3,\alpha_3},\Ad (k_s^{-1}k_0)H')=0.\]
	Then we show that $g_{3,3}:S_3\to \R$ is a submersion at $(0,0;k_0)$ by showing $[v_{3,3}g_{3,3}](0,0;k_0)$ not 0.

	\[\begin{split}
		&\quad[v_{3,3}g_{3,3}](0,0;k_0)\\
		&=B([\Ad k_s(X_{3,\alpha_3}-\theta X_{3,\alpha_3}),\Ad k_s X_3],\Ad k_0 H'_1)\\
		&=B(\Ad k_s(X_{3,\alpha_3}-\theta X_{3,\alpha_3}),[\Ad k_s X_3,\Ad (k_1k_s)H'_1])\\
		&=-B(X_{3,\alpha_3}-\theta X_{3,\alpha_3},[\Ad (k_s^{-1}k_1k_s)H'_1,X_3])\\
		&=-\alpha_3(\Ad (k_s^{-1}k_1k_s) H'_1)\\
		&=-\epsilon_3\alpha_3(H'_1).
	\end{split}\]
According to our choice of $\alpha_3$, $\alpha_3(H'_1)$ is not 0.

	Let $S_4$ be $\{(b_2,c_2;k_0)\in S_3\}$ which corresponds to the level set $g_{3,3}^{-1}(\{0\})$. Follow the procedure in the proof of Theorem \ref{thm-n1}. We get a local fibration at $(0,0;k_0)$, $\kappa_3:S_3\to S_4\times\R$ and $\pi_3:S_3\to S_4, (b_2,c_2;k)\to (b_2,c_2;k_0)$. It is not hard to get another local fibration at $(0,0;k_0)$, $\cka_3:S_3\to S_4\times \R, s_3\to (s_4,\ct_3)$ such that
	\[[(f_3-f_3\circ\pi_3)\circ \cka_3^{-1}](s_4,\ct_3)=-\frac{1}{2}\epsilon_3s\alpha_3(H)\alpha_3(H')\ct_3^2.\]
	Then we obtain the third local coordinate and conclude the proof as in the Theorem \ref{thm-n1}.  
\end{proof}

Before we discuss the next case, we give a heuristic remark for Theorem \ref{thm-n1} and Theorem \ref{thm-n2}. 

\begin{remark}\label{rem-ori}
Let $k_0\in \rM'$. The map $\fk\to \rK: X\to k_0 \exp X$ is a local diffeomorphism at $0\in \fk$. For $H,H'\in \fa$, we consider $f_{H,H'}(k_0\exp{X})$ and have
	\[\begin{split}
						&\quad B(H,\Ad (k_0\exp{X})H')\\
					       &=B(\Ad(k_0^{-1})H,\exp(\ad X)H')\\
					       &=\sum_{n=0}^\infty \frac{1}{n!}B(\Ad(k_0^{-1}H),(\ad X)^n H')\\
					       &=B(\Ad(k_0^{-1})H,H')+\sum_{n=2}^{\infty} \frac{1}{n!} B(\Ad (k_0^{-1})H,(\ad X)^n H'),
	\end{split}\]
	here the term for $n=1$ is dropped because $B(\Ad(k_0^{-1})H,[X,H'])=0$. Let $\alpha_1,\alpha_2,\alpha_3$ be the three positive roots arranged in any order. Let $\{X_1,X_2,X_3\}$ be an orthonormal basis of $\fk$ such that $X_i\in\fk_{\alpha_i}$ for $i=1,2,3$. Then $X$ could be written as $x_1X_1+x_2X_2+x_3X_3$, here $x_1,x_2,x_3\in \R$. With the coordinates, we have
	\[\begin{split}
	 	&\quad \sum_{n=2}^{\infty} \frac{1}{n!} B(\Ad (k_0^{-1})H,(\ad X)^n H')\\
		&=-\frac{1}{2}B([\Ad (k_0^{-1})H,X],[H',X])+\sum_{n=3}^{\infty} \frac{1}{n!} B(\Ad (k_0^{-1})H,(\ad X)^n H').\\
	\end{split}\]
Note that 
	\[-\frac{1}{2}B([\Ad (k_0^{-1})H,X],[H',X])=-\frac{1}{2}\sum_{i=1}^3 \alpha_i(\Ad k_0^{-1}H)\alpha_i(H')x_i^2.\]	
Hence Theorem \ref{thm-n1} and Theorem \ref{thm-n2} can be interpreted in the following way that there exists a family of local diffeomorphisms at $0\in \fk$ depending smoothly on the parameters $H,H'$, $\fk\to \fk:X=\sum_{i=1}^3 x_iX_i\to \tilde{X}=\sum_{i=1}^3 \cx_iX_i$, such that
	 \[\sum_{n=2}^{\infty} \frac{1}{n!} B(\Ad (k_0^{-1})H,(\ad X)^n H')=\frac{1}{2}B(\Ad(k_0^{-1})H,(\ad \tilde{X})^2 H').\]	
\end{remark}
	
\begin{theorem} \label{thm-n5}
Assume $k_0$ is contained in the critical manifold $\rK_{H_1}s\rK_{H'_1}$ which has dimension 2. Let $\alpha_1,\alpha_2,\alpha_3$ be arranged as in Theorem \ref{thm-n2}.

 If $k_0\in \rK_{H_1}s\rK_{H'_1}\backslash (\rM'_{H_1}s\rK_{H'_1}\cup \rK_{H_1}s\rM'_{H'_1})$, then there exists a local coordinate system at $(0,0;k_0)$, $\kappa:S_1\to \R^5$ which preserves the parameters. $\kappa$ satisfies
	\[\begin{split}
		f_1\circ \kappa^{-1}(b_2,c_2;\cx_1,\cx_2,\cx_3)&=f_1\circ\kappa^{-1}(b_2,c_2;0,0,0)-\frac{1}{2}s\alpha_1(H)\alpha_1(H')\cx_1^2\\
		&-s\alpha_2(H)\alpha_2(H')\cx_2-s\alpha_3(H)\alpha_3(H')\cx_3,
	\end{split}\]
here $H=H_1+b_2H_2,H'=H'_1+c_2H'_2$.
\end{theorem}

\begin{proof}
According to the assumption  \[k_0=k_1k_sk_2\in \rK_{H_1}s\rK_{H'_1}\backslash (\rM'_{H_1}s\rK_{H'_1}\cup \rK_{H_1}s\rM'_{H'_1}),\] 
we have $k_1\notin \rM'_{H_1}$ and $k_2\notin \rM'_{H'_1}$.

Let $\{X_1,X_2,X_3\}$ be an orthonormal basis of $\fk$ such that $X_i\in \fk_{\alpha_i}$ for $i=1,2,3$. We have $[X_2,H'_1]=0$ and $[\Ad (k_s)X_3,H_1]=0$. Let $v_1$ be any vector field on $\rK$ such that $v_1(k_0)=\Ad(k_2^{-1})X_1\in \fk=\tg{\rK}{k_0}$. Let $v_2$ be the left invariant vector field $\inv{X_2}{L}$ and let $v_3$ be the right invariant vector field $\inv{(\Ad (k_s)X_3)}{R}$. Let $v_{1,1},v_{1,2},v_{1,3}$ be the canonical extensions of $v_1,v_2,v_3$ to $S_1=\R\times \R\times \rK$.

The construction of the first local coordinate $\cx_1:S_1\to \R$ resembles the one in the proof of Theorem \ref{thm-n1}. 

Set $g_{1,1}=v_{1,1}f_1:S_1\to \R$. We claim that $g_{1,1}$ vanishes at $(0,0;k_0)\in S_1$ and $g_{1,1}$ is a submersion at $(0,0;k_0)$. The first claim is obvious since $k_0$ is a critical point of $f_{H_1,H'_1}$. 
	The second claim is true for $[v_{1,1}g_{1,1}](0,0;k_0)\neq 0$. We have $v_{1,1}g_{1,1}=v_{1,1}^2f_1:S_1\to \R$. As $v_{1,1}$ is a vertical vector field on $S_1$, we have 
		\[[v_{1,1}^2f_1](0,0;k_0)=[v_{1}^2f_{H_1,H'_1}](k_0).\]
	Although the explicit expression of $v_1$ is not available, the fact that $k_0$ is a critical point of $f_{H_1,H'_1}$ allows us to replace $v_1$ with the left invariant vector field $\inv{[v_1(k_0)]}{L}$ when evaluating this second order derivative at $k_0$,
		\[\begin{split}
			[v_{1}^2f_{H_1,H'_1}](k_0)&=[\inv{(\Ad k_2^{-1}X_1)}{L}\inv{(\Ad k_2^{-1}X_1)}{L}f_{H_1,H'_1}](k_0)\\
			&=B(H_1,\Ad k_0[\Ad k_2^{-1}X_1,[\Ad k_2^{-1} X_1,H'_1]])\\
			&=B(s^{-1}H_1,[X_1,[X_1,H'_1]])\\
			&=-s\alpha_1(H_1)\alpha_1(H'_1).
		\end{split}\]
	According to the choice of $\alpha_1$, $s\alpha_1(H_1)\alpha_1(H'_1)\neq 0$.
	
	We follow the routine listed in the proof of Theorem \ref{thm-n1}. We take the submanifold $S_2$ through $(0,0;k_0)$ corresponding to the level set $g_{1,1}^{-1}(\{0\})$. We have that $v_{1,1}$ is transverse to $S_2$ at $(0,0;k_0)$. The flow generated by $v_{1,1}$ induces a local fibration at $(0,0;k_0)$, $\kappa_1:S_1\to S_2\times \R$ and we define $\pi_1:S_1\to S_2$ accordingly. The Taylor's formula gives
	\begin{eqnarray*}
	&&\left[(f_1-f_1\circ \pi_1)\circ \kappa_1^{-1}\right](s_2,t_1)\\
	&&\quuad=\frac{1}{2}t_1^2\int_0^1 2(1-z)[(v_{1,1}g_{1,1})\circ \kappa_1^{-1}](s_2,t_1z)dz.
	\end{eqnarray*}
	We set
		\[\ct_1(s_2,t_1)=t_1\Big[\frac{-1}{s\alpha_1(H)\alpha_1(H'_1)}\int_0^1 2(1-z)[(v_{1,1}g_{1,1})\circ \kappa_1^{-1}](s_2,t_1z)dz \Big]^{1/2},\]
	then get a local fibration $\cka_1:S_1\to S_2\times \R, s_1\to (s_2,\ct_1)$ such that
		\[f_1(b_2,c_2;k)=f_1\circ \pi_1(b_2,c_2;k)-\frac{1}{2}s\alpha_1(H)\alpha_1(H')\ct_1^2.\]
		
	Next we construct the second and third local coordinate $\cx_2,\cx_3:S_1\to \R$ simultaneously.
	
	Let $v_{2,2},v_{2,3}$ be the projections of $v_{1,2},v_{1,3}$ to $S_2$ with respect to $v_{1,1}$. They are both vertical vector fields on $S_2$.
	
	We consider the functions $v_{1,2}f_1,v_{1,3}f_1:S_1\to \R$, for $s_1\in S_1$
	\[[v_{1,2}f_1](s_1)=\mu_{1,2}(s_1)g_{1,2}(s_1),\]
	\[[v_{1,3}f_1](s_1)=\mu_{1,3}(s_1)g_{1,3}(s_1),\]
	here $\mu_{1,2},\mu_{1,3},g_{1,2},g_{1,3}:S_1\to \R$ have exactly same expressions \eqref{eq:4-1},\eqref{eq:4-2},\eqref{eq:4-3} and \eqref{eq:4-4} in Theorem \ref{thm-n2}.
	
	Let $f_2:S_2\to \R$ be the restriction of $f_1:S_1\to \R$ to $S_2$. We also restrict $\mu_{1,2},\mu_{1,3},g_{1,2},g_{1,3}:S_1\to \R$ to $S_2$ and denote these restrictions by $\mu_{2,2},\mu_{2,3},g_{2,2},g_{2,3}:S_2\to \R$. Since $v_{1,1}f_1:S_1\to \R$ vanishes on $S_2$, Lemma \ref{lem-id} gives that for $s_2\in S_2$
		\[[v_{2,2}f_2](s_2)=[v_{1,2}f_1](s_2),\quad [v_{2,3}f_2](s_2)=[v_{1,3}f_1](s_2).\]
We also have that for $s_2\in S_2$
	\[[v_{2,2}f_2](s_2)=\mu_{2,2}(s_2)g_{2,2}(s_2),\quad [v_{2,3}f_2](s_2)=\mu_{2,3}(s_2)g_{2,3}(s_2).\]
	
	We want to show that
		\[g_{2,2}(0,0;k_0)\neq 0, \quad g_{2,3}(0,0;k_0)\neq 0.\]
	Evaluate	
	 \[\begin{split}
	 g_{2,2}(0,0;k_0)&=B(H_1,\Ad k_0(X_{2,\alpha_2}-\theta X_{2,\alpha_2}))\\
	 &=B(s^{-1}H_1,\Ad k_2(X_{2,\alpha_2}-\theta X_{2,\alpha_2})).
	 \end{split}\]
	 Since
	  \[
	 	s^{-1}H_1-\alpha_2(s^{-1}H_1)\alpha_2^{\vee}/\len{\alpha_2^\vee}^2\in \ker \alpha_2=\R H'_1,\]
	 	\[B(H'_1,\Ad k_2(X_{2,\alpha_2}-\theta X_{2,\alpha_2}))=0,\] we have
	 	\[ g_{2,2}(0,0;k_0)=\frac{\alpha_2(s^{-1}H_1)}{\len{\alpha_2^\vee}^2}B(\alpha_2^\vee,\Ad k_2(X_{2,\alpha_2}-\theta X_{2,\alpha_2})).\]
	 As $k_2\notin \rM'_{H'_1}$ and $\alpha_2(H_1') = 0$, Lemma \ref{lem-v2} tells that $g_{2,2}(0,0;k_0)$ is not 0.

	 $g_{2,3}(0,0;k_0)$ is not 0 for the similar reason. 
	 \[\begin{split}
	 	g_{2,3}(0,0;k_0)&=B(\Ad (k_0^{-1}k_s) (X_{3,\alpha_3}-\theta X_{3,\alpha_3}),H'_1)\\
	 						&=B(\Ad (k_s^{-1}k_1^{-1}k_s)(X_{3,\alpha_3}-\theta X_{3,\alpha_3}),H'_1)\\
	 						&=\frac{\alpha_3(H'_1)}{\len{\alpha_3^\vee}^2}B(\alpha_3^\vee,\Ad (k_s^{-1}k_1^{-1}k_s)(X_{3,\alpha_3}-\theta X_{3,\alpha_3}))
	\end{split},\]
	here $k_s^{-1}k_1^{-1}k_s\in k_s^{-1}\rK_{H_1}k_s=\rK_{s^{-1}H_1}$ and $\alpha_3(H'_1)\neq 0$. Since $k_1\notin \rM'_{H'_1}$, $ k_s^{-1}k_1k_s\notin \rM'_{s^{-1}H_1}$. Lemma \ref{lem-v1} tells that $g_{2,3}(0,0;k_0)\neq 0$.
	
	Let $\eta_2,\eta_3:S_2\times \R \to S_2$ be the flows (or local 1-parameter groups) of $v_{2,2},v_{2,3}$. The subset $S_4 = \{\pi_1(b_2,c_2;k_0)\in S_2|(b_2,c_2;k_0)\in S_1\}$  could be viewed as a 2-dimension submanifold of $S_2$. We would like to construct a local fibration at $(0,0;k_0)$, $S_2\to S_4\times \R^2$ with the aid of $\eta_2,\eta_3,S_4$.
	
	We consider the map 
		\[S_4\times \R\times \R\to S_2,\quad (s_4,t_3,t_2)\to \eta_2(\eta_3(s_4,t_3),t_2).\]
	We show that it is a local diffeomorphism at $((0,0;k_0),0,0)$. We only need to show that $\{v_{2,2},v_{2,3}\}$ is transverse to $S_4$ at $s_1=(0,0;k_0)$. First, we have
		\[\tg{S_1}{s_1}=\R\frac{\partial}{\partial b_2}(s_1)+\R \frac{\partial }{\partial c_2}(s_1)+\R v_{1,1}(s_1)+\R v_{1,2}(s_1)+\R v_{1,3}(s_1).\]
		Take the projection $\pr{\cdot}{S_1}{S_2}:\tg{S_1}{s_1}\to \tg{S_2}{s_1}$ with respect to $v_{1,1}(s_1)$, then we have 
		\[\tg{S_2}{s_1}=\R\pr{\frac{\partial}{\partial b_2}(s_1)}{S_1}{S_2}+\R \pr{\frac{\partial }{\partial c_2}(s_1)}{S_1}{S_2}+0+\R v_{2,2}(s_1)+\R v_{2,3}(s_1).\]
		According to the definition of $S_4$, we have
			\[\tg{S_4}{s_1}=\R d\pi_1\left(\frac{\partial}{\partial b_2}(s_1)\right)+\R d\pi_1\left( \frac{\partial}{\partial c_2}(s_1)\right).\]
		By Lemma \ref{lem-equi}, we have
			\[\tg{S_2}{s_1}=\tg{S_4}{s_1}+\R v_{2,2}(s_1)+\R v_{2,3}(s_1),\]
		that is, $\{v_{2,2},v_{2,3}\}$ is transverse to $S_4$ at $(0,0;k_0)$.
		
		For $(s_4,t_3,t_2)\in S_4\times \R\times \R$, we have
			\[\begin{split}
				f_2(\eta_2(\eta_3(s_4,t_3),t_2))&=\Big [[f_2\circ \eta_2](\eta_3(s_4,t_3),t_2)-[f_2\circ \eta_2](\eta_3(s_4,t_3),0)\Big]\\
				&+\Big[[f_2\circ\eta_3](s_4,t_3)-[f_2\circ \eta_3](s_4,0)\Big]+f_2(s_4)
			\end{split}.\]
		The term in the first bracket is equal to
		\begin{eqnarray*}
		&&t_2\int_0^1 \partial_{t_2}[f_2\circ \eta_2](\eta_3(s_4,t_3),t_2z)dz\\
		&&\quuad\quad\quad =\mu_{2,2}(s_4)t_2\int_0^1[g_{2,2}\circ \eta_2](\eta_3(s_4,t_3),t_2z)dz.
		\end{eqnarray*}
		The term in the second bracket is equal to
		\[ t_3\int_0^1 \partial_{t_3}[f_2\circ \eta_3](s_4,t_3z)dz=\mu_{2,3}(s_4)t_3\int_0^1[g_{2,3}\circ \eta_3](s_4,t_3z)dz.\]
		Here the factors $\mu_{2,2},\mu_{2,3}$ are moved outside the integrals as in the proof of Theorem \ref{thm-n1}.
		
		We introduce a local diffeomorphism at $((0,0;k_0),0,0)$, $S_4\times \R\times \R\to S_4\times \R\times \R, (s_4,t_3,t_2)\to (s_4,\ct_3,\ct_2)$ by putting
		\[\ct_2(s_4,t_3,t_2)=t_2\cdot \frac{1}{s\alpha_2(H)}\int_0^1[g_{2,2}\circ \eta_2](\eta_3(s_4,t_3),t_2z)dz,\]
		\[\ct_3(s_4,t_3,t_2)=t_3\cdot \frac{-1}{\alpha_3(H')}\int_0^1[g_{2,3}\circ \eta_3](s_4,t_3z)dz.\]
		
	Take the composition of the inverse of the map $(s_4,t_3,t_2)\to \eta_2(\eta_3(s_4,t_3),t_2)$ and the map above, we have a local fibration $\cka_2:S_2\to S_4\times \R\times \R, s_2\to (s_4,\ct_3,\ct_2)$ satisfying 
		\[f_2\circ \cka_2^{-1}(s_4,t_3,t_2)=f_2(s_4)-s\alpha_2(H)\alpha_2(H')\ct_2-s\alpha_3(H)\alpha_3(H')\ct_3.\]

		We take the local composition to get the second and third local coordinate
		\[S_1\xrightarrow{\cka_1}S_2\times \R \xrightarrow{(\cka_2,\id)}S_4\times \R^2\times \R,\]
		we have
			\[\begin{split}f_1(b_2,c_2;k)&=f_1(\pi_1(b_2,c_2;k_0))-\frac{1}{2}s\alpha_1(H)\alpha_1(H')\cx_1^2\\
		&-s\alpha_2(H)\alpha_2(H')\cx_2-s\alpha_3(H)\alpha_3(H')\cx_3.
			\end{split}\]
			
	As $\R^2\to S_4: (b_2,c_2)\to \pi_1(b_2,c_2;k_0)$ is a local diffeomorphism at $(0,0)$, we conclude the proof by including the two parameters $b_2,c_2$ as the local coordinates.
\end{proof}

\begin{corollary}
Assume $k_0$ is contained in the critical manifold $\rK_{H_1}s\rK_{H'_1}$ which has dimension 2. Let the positive roots $\alpha_1,\alpha_2,\alpha_3$ be arranged as in Theorem \ref{thm-n2}. 

If $k_0\in \rK_{H_1}s\rM'_{H'_1}\backslash \rM'_{H_1}s\rK_{H'_1}$, then there exists a local coordinate system at $(0,0;k_0)$, $\kappa:S_1\to \R^5$ which preserves the parameters. $\kappa$ satisfies 
	\[\begin{split}
	 	f_1\circ \kappa^{-1}(b_2,c_2;\cx_1,\cx_2,\cx_3)&=f_1(b_2,c_2;k_0)-\frac{1}{2}s\alpha_1(H)\alpha_1(H')\cx_1^2\\
											&-\frac{1}{2}\epsilon\cdot s\alpha_2(H)\alpha_2(H')\cx_2^2-s\alpha_3(H)\alpha_3(H')\cx_3,
	\end{split}\]
	here $H=H_1+b_2H_2,H'=H'_1+c_2H'_2$, $\epsilon$ is a constant $\pm 1$ which is determined explicitly. Besides, $\kappa$ always maps $(b_2, c_2; k_0)$ to $(b_2, c_2; 0, 0, 0)$.
\end{corollary}

\begin{proof}
According to the assumption $k_0=k_1k_sk_2\in\rK_{H_1}s\rM'_{H'_1}\backslash \rM'_{H_1}s\rK_{H'_1}$, we have $k_1\in \rK_{H_1}\backslash \rM'_{H_1}$ and $k_2\in \rM'_{H'_1}$.

Let $\{X_1, X_2, X_3\}$ be an orthonormal basis of $\fk$ such that $X_i\in\fk_{\alpha_i}$, for $i=1,2,3$. Let $v_1=\inv{(\Ad(k_2^{-1})X_1)}{L},v_2=\inv{X_2}{L},v_3=\inv{(\Ad(k_s)X_3)}{R}$. Let $v_{1,1},v_{1,2},v_{1,3}$ be their canonical extensions to $S_1$.

Consider $v_{1,1}f_1,v_{1,2}f_1:S_1\to \R$,
\[[v_{1,1}f_1](s_1)=\mu_{1,1}(s_1)g_{1,1}(s_1), \quad [v_{1,2}f_1](s_1)=\mu_{1,2}(s_1)g_{1,2}(s_1).\]
The formulas for $\mu_{1,1},g_{1,1},\mu_{1,2},g_{1,2}:S_1\to \R$ are given by equation \eqref{eq:4-1}, \eqref{eq:4-2}, \eqref{eq:4-3} and \eqref{eq:4-4}. We know $g_{1,1},g_{1,2}$ vanishes at $\R\times\R\times \rK_{H_1}k_0$ for the same reason in the proof of Theorem \ref{thm-n2}. With slight modifications of the proof of Theorem \ref{thm-n2} and Corollary \ref{cor-n4}, we establish this corollary.
\end{proof}

\begin{corollary}
Assume $k_0$ is contained in the critical manifold $\rK_{H_1}s\rK_{H'_1}$ which has dimension 2. Let the positive roots $\alpha_1,\alpha_2,\alpha_3$ be arranged in the way such that $s\alpha_2(H_1)=\alpha_3(H'_1)=0$. 

If $k_0\in \rM'_{H_1}s\rK_{H'_1}\backslash \rK_{H_1}s\rM'_{H'_1}$, then there exists a local coordinate system at $(0,0;k_0)$, $\kappa:S_1\to \R^5$ which preserves the parameters. $\kappa$ satisfies 
	\[\begin{split}
	 	f_1\circ \kappa^{-1}(b_2,c_2;\cx_1,\cx_2,\cx_3)&=f_1(b_2,c_2;k_0)-\frac{1}{2}s\alpha_1(H)\alpha_1(H')\cx_1^2\\
											&-\frac{1}{2}\epsilon\cdot s\alpha_2(H)\alpha_2(H')\cx_2^2-s\alpha_3(H)\alpha_3(H')\cx_3,
	\end{split}\]
	here $H=H_1+b_2H_2,H'=H'_1+c_2H'_2$, $\epsilon$ is a constant $\pm 1$ which is determined explicitly. Besides, $\kappa$ always maps $(b_2,c_2;k_0)$ to $(b_2,c_2;0,0,0)$.
\end{corollary}

\begin{proof}
According to the assumption $k_0=k_1k_sk_2\in\rM'_{H_1}s\rK_{H'_1}\backslash \rK_{H_1}s\rM'_{H'_1}$, we have $k_1\in \rM'_{H_1}$ and $k_2\in \rK_{H'_1}\backslash\rM'_{H'_1}$. 

Let $\{X_1,X_2,X_3\}$ be an orthonormal basis of $\fk$ such that $X_i\in\fk_{\alpha_i}$, for $i=1,2,3$. Let $v_1=\inv{(\Ad(k_1k_s)X_1)}{R},v_2=\inv{(\Ad(k_s)X_2)}{R},v_3=\inv{X_3}{L}$. Let $v_{1,1},v_{1,2},v_{1,3}$ be their canonical extensions to $S_1$.

Consider $v_{1,1}f_1,v_{1,2}f_1:S_1\to \R$,
\[[v_{1,1}f_1](s_1)=\mu_{1,1}(s_1)g_{1,1}(s_1),\]
\[\mu_{1,1}(b_2,c_2;k)=s\alpha_1(H_1+b_2\Ad (k_1^{-1})H_2),\]
\[g_{1,1}(b_2,c_2;k)=B(X_{1,\alpha_1}-\theta X_{1,\alpha_1},\Ad(k_s^{-1}k_1^{-1}k)(H'_1+c_2H'_2)),\]
\[[v_{1,2}f_1](s_1)=\mu_{1,2}(s_1)g_{1,2}(s_1),\]
\[\mu_{1,2}(b_2,c_2;k)=s\alpha_2(H_1+b_2H_2),\]
\[g_{1,2}(b_2,c_2;k)=B(X_{2,\alpha_2}-\theta X_{2,\alpha_2},\Ad(k_s^{-1}k)(H'_1+c_2H'_2)).\]
We can show that $g_{1,1},g_{1,2}$ vanishes at $\R\times\R\times k_0\rK_{H'_1}$ by Lemma \ref{lem-v1}. With slight modifications of the proof of Theorem \ref{thm-n2} and Corollary \ref{cor-n4}, we establish this corollary.
\end{proof}

\subsection{Either $H_1$ or $H'_1$ is regular}
We have finished the discussions about the cases when both $H_1, H'_1$ are singular. Now we turn to the cases when either $H_1$ or $H'_1$ is regular. They more or less follow immediately from the cases when both $H_1, H'_1$ are singular.

\begin{corollary}
$H_1$ is regular and $H'_1$ is singular. Assume $k_0$ is contained in the critical manifold $s\rK_{H'_1}$, $s\in \rW$. We arrange the postive roots $\alpha_1,\alpha_2,\alpha_3$ in the way such that $\alpha_3$ is the unique positive root with $\alpha_3(H'_1)=0$.

If $k_0\in s\rK_{H'_1}\cap \rM'$, then there exists a local coordinate system at $(0,0;k_0)$, $\kappa:S_1\to \R^5$ which  satisfies

	\[
		\begin{split}
			f_1\circ \kappa^{-1}(b_2, c_2; \cx_1, \cx_2, \cx_3)&=f_1(b_2, c_2; k_0)-\frac{1}{2}s\alpha_1(H)\alpha_1(H')\cx_1^2\\
			&-\frac{1}{2}s\alpha_2(H)\alpha_2(H')\cx_2^2-\frac{1}{2}\epsilon \cdot s\alpha_3(H)\alpha_3(H')\cx_3^2,
		\end{split}
	\]
	here $\epsilon$ is a constant $\pm 1$.
	
If $k_0\in s\rK_{H'_1}\backslash \rM'$, then there exists a local coordinate system at $(0,0;k_0)$, $\kappa:S_1\to \R^5$ which satisfies
	\[\begin{split}
	f_1\circ \kappa^{-1}(b_2,c_2;\cx_1,\cx_2,\cx_3)&=f_1(b_2,c_2;k_0)-\frac{1}{2}s\alpha_1(H)\alpha_1(H')\cx_1^2\\
	&-\frac{1}{2}s\alpha_2(H)\alpha_2(H')\cx_2^2
	-s\alpha_3(H)\alpha_3(H')\cx_3.
	\end{split}\]	
	Here $H=H_1+b_2H_2, H'=H'_1+c_2H'_2$. In both cases, $\kappa$ preserves the parameters and always maps $(b_2,c_2;k_0)$ to $(b_2,c_2;0,0,0)$.
\end{corollary}

\begin{proof}
The proofs of Theorem \ref{thm-n1} and Corollary \ref{cor-n4} cover the cases here. 
\end{proof}

The situation for singular $H_1$ and regular $H'_1$ is similar. We simply replace $s\rK_{H'_1}$ by $\rK_{H_1}s$ and replace $\alpha_3(H'_1)=0$ by $s\alpha_3(H_1)=0$ in the statement above and all else remain the same.

When $H_1,H'_1$ are both regular, $f_{H_1,H'_1}$ is a Morse function of which the critical set is $\rM'$. When $|b_2|,|c_2|$ are  small, $f_{H,H'}$ remains to be a Morse function which shares the same critical set $\rM'$, here $H=H_1+b_2H_2,H'=H'_2+c_2H'_2$.

\begin{proposition}
$H_1,H'_1$ are regular. Assume $k_0$ is a critical point of $f_{H_1,H'_1}$. As $k_0\in \rM'$, set $s=k_0\rM\in \rW$. Let $\alpha_1,\alpha_2,\alpha_3$ be the positive roots arranged in any order. There exists a local coordinate system at $(0,0;k_0)$, $\kappa:S_1\to \R^5$ which preserves the parameters. $\kappa$ satisfies
	\[
		f_1\circ \kappa^{-1}(b_2,c_2;\cx_1,\cx_2,\cx_3)=f_1(b_2,c_2;k_0)-\frac{1}{2}\sum_{i=1}^3 s\alpha_i(H)\alpha_i(H')\cx_i^2,
	\]
	here $H=H_1+b_2H_2, H'=H'_1+c_2H'_2$.
	Besides, $\kappa$ always maps $(b_2,c_2;k_0)$ to $(b_2,c_2;0,0,0)$.
\end{proposition}

\begin{proof}
The proofs of Theorem \ref{thm-n1} and Theorem \ref{thm-n2} cover this situation. Or we can apply the Morse lemma with parameters directly, cf. Hörmander \cite[Appendix C.6]{H}.
\end{proof}

\section{Estimates for oscillatory integrals}

In this section, we first establish the estimate for the oscillatory integral \eqref{eq:ost2}, then use it to deduce the estimate for the oscillatory integral \eqref{eq:ost1}. 

\subsection{Estimate for oscillatory integral \eqref{eq:ost2}} 
\begin{theorem} \label{est-1}
We have
\[\left |\int_\rK u(k) \exp {\iu B(H,\Ad (k)H')}dk\right |\lesssim \len{u}_{C^3(\rK,\C)}\sum_{s\in \rW}\Omega(s^{-1}H,H')^{-1/2},\]
\[\quuad \forall H,H'\in \fa, u\in C^\infty(\rK,\C).\]
$\Omega$ is defined as in Theorem \ref{est-sphfun}.
\end{theorem}

We point out that the norm $\len{\cdot}_{C^3(\rK,\C)}$ is not canonical. It is defined with the aid of a fixed finite cover of local coordinate patches and a fixed finite $\smo$ partition of unity. Here the compactness of $\rK$ is used. Throughout our discussions, we assume this fixed norm. 

We now come to the proof of Theorem \ref{est-1}, the global estimate for the oscillatory integral \eqref{eq:ost2}. We make the following two reductions.

\begin{lemma} \label{lem-loc1}
For any $k_0\in \rK$, there exists an open neighborhood $U$ of $k_0$ such that for any $\chi\in \smo_c(U,\C)$, we have
\[\left|\int_\rK u(k)\chi(k) e^{\iu f_{H,H'}(k)}dk\right|\lesssim_{U,\chi} \len{u}_{C^3(\rK,\C)} \sum_{s\in \rW}\Omega(s^{-1}H,H')^{-1/2}\]
\[\quuad \forall H,H'\in \fa,u\in C^\infty(\rK,\C).\]
\end{lemma}

It is not difficult to see that the lemma above implies Theorem \ref{est-1}. For $k_0\in \rK$, let $U_{k_0}$ be the open neighborhood of $k_0$ chosen according to Lemma \ref{lem-loc1}. Then these open subsets form an open cover of $\rK$. Due to the compactness of $\rK$, this cover has a finite subcover denoted by $\{U_i\}_{i=1}^n$. Let $\{\chi_i\}_{i=1}^n$ be a $\smo$ partition of unity subordinate to the open cover $\{U_i\}_{i=1}^n$. Then we break the whole integral into several pieces 
	\[\int_\rK u(k) e^{\iu f_{H,H'}(k)}dk=\sum_{i=1}^n \int_\rK u(k)\chi_i(k) e^{\iu f_{H,H'}(k)}dk.\]
As Lemma \ref{lem-loc1} gives the estimate of each term in the sum, we obtain the estimate for the whole integral. 

The first reduction is made to localize the integral at various points, in other words, decompose the integral into integrals over sufficiently small regions. We make the second reduction to restrict the parameters $(H,H')$ to small conic subsets of $\fa\times\fa$.

\begin{lemma} \label{lem-loc2}
For any $k_0\in \rK$, $H_1,H'_1\in \fa$ with $\len{H_1}=\len{H'_1}=1$, there exists an open neighborhood $U$ of $k_0$ and $\varepsilon>0$ such that for any $\chi\in \smo_c(U,\C)$, we have
	\[\left|\int_\rK u(k)\chi(k) e^{\iu f_{H,H'}(k)}dk\right|\lesssim_{U,\varepsilon,\chi} \len{u}_{C^3} \sum_{s\in \rW}\Omega(s^{-1}H,H')^{-1/2},\quad \forall u\in C^\infty(\rK,\C),\]	
	\[H\in \{b_1(H_1+b_2H_2)\in \fa| b_1\in \R, |b_2|<\varepsilon\},\] 
	\[H'\in \{c_1(H'_1+c_2H'_2)\in \fa| c_1\in \R, |c_2|<\varepsilon\}.\]
Here $H_2,H'_2$ are chosen so that $\{H_1,H_2\}$, $\{H'_1,H'_2\}$ are orthonormal bases of $\fa$.


\end{lemma}

It is not hard to see that Lemma \ref{lem-loc2} implies Lemma \ref{lem-loc1}. Fix $k_0\in \rK$. For $(H_1,H'_1)\in \fa\times\fa$ with $\len{H_1}=\len{H'_1}=1$, let $U_{H_1,H'_1},\varepsilon_{H_1,H'_1}$ be chosen according to Lemma \ref{lem-loc2}. Thus we get a cover of $\fa\times\fa$ by the conic subsets 
	\begin{eqnarray*}
	&& \{(b_1(H_1+b_2H_2),c_1(H'_1+c_2H'_2))\in\fa\times\fa|\\
	&&\quuad b_1,c_1\in \R, |b_2|<\varepsilon_{H_1,H'_1},|c_2|<\varepsilon_{H_1,H'_1}\}.
	\end{eqnarray*}
Then there exists a finite subcover by a compactness argument. The intersection of those $U_{H_1,H'_1}$'s in this finite subcover gives the open neighborhood of $k_0$ in Lemma \ref{lem-loc1}.

Before we prove Lemma \ref{lem-loc2}, we give a uniform estimate for the oscillatory integral of a specific type. It could be regarded as a version of van der Corput lemma (cf. Sogge \cite[Lemma 1.1.2]{Sog}) in higher dimensions.

\begin{lemma} \label{lem-os}
Let $d$ be a positive integer. For any $L>0$,  we have
$$
\left|\int_{\R^d} u(x) \exp \frac{\iu}{2} (t_1x_1^2+\dots+t_d x_d^2)dx_1\dots dx_d \right|\lesssim_{d,L}\len{u}_{C^d}\prod_{i=1}^d(1+|t_i|)^{-1/2},
$$
\[\quuad \forall (t_1,\dots,t_d)\in \R^d, u\in \smo_c((-L,L)^d).\]
\end{lemma}

\begin{proof}
We prove the lemma by induction. Let $I$ denote the interval $(-L,L)$. 

When $d=1$, it is an easy consequence of Stein \cite[Corollary, Page 334]{Ste}.

Now we assume the statement holds for $d-1$. Then we prove the statement for $d$. 

We take the partial integral with the first $d-1$ variables,
\[
\begin{split}
\star&=\int_{\R^d} u(x) \exp \frac{\iu}{2}(t_1x_1^2+\dots+t_d x_d^2)dx_1\dots dx_d\\
&=\int_I \tilde{u}(t_1,\dots,t_{d-1};x_d) \exp \frac{\iu}{2} (t_dx_d^2) dx_d,
\end{split}
\]
where $\tilde{u}(t_1,\dots,t_{d-1};x_d)$ is
$$
\int_{\R^{d-1}} u(x) \exp \frac{\iu}{2}(t_1x_1^2+\dots+t_{d-1} x_{d-1}^2)dx_1\dots dx_{d-1}.
$$
It is important that the term $\tilde{u}(t_1,\dots,t_{d-1};\cdot)$ is viewed as a function in $x_d$ and it belongs to $\smo_c(I)$. Then the statement for $d=1$ yields
\[|\star|\lesssim_L\left\|\tilde{u}(t_1,\dots,t_{d-1};\cdot)\right\|_{C^1(I)}(1+|t_d|)^{-1/2}.\]
By definition,
	\begin{eqnarray*}
	&&\len{\tilde{u}(t_1,\dots,t_{d-1};\cdot)}_{C^1(I)}\\
	&&\quad\quad\quad =\max \left\{\sup_{x_d\in \R}|\tilde{u}(t_1,\dots,t_{d-1};x_d)|,\sup_{x_d\in \R}\left|\partial_{x_d}\tilde{u}(t_1,\dots,t_{d-1};x_d)\right|\right\}.
	\end{eqnarray*}

According to the statement for $d-1$, we have
	\[|\tilde{u}(t_1,\dots,t_{d-1};x_d)|\lesssim_{d-1,L}\len{u(\cdot,x_d)}_{C^{d-1}(I^{d-1})}\prod_{i=1}^{d-1}(1+|t_i|)^{-1/2},\]
	here $u(\cdot,x_d)$ with fixed $x_d$ is viewed as a function in $\smo_c(I^{d-1})$. Hence,
	\[\sup_{x_d\in \R} |\tilde{u}(t_1,\dots,t_{d-1};x_d)|\lesssim_{d-1,L}\len{u}_{C^{d-1}(I^d)}\prod_{i=1}^{d-1}(1+|t_i|)^{-1/2}\]
	for
	$\sup_{x_d\in \R} \len{u(\cdot,x_d)}_{C^{d-1}(I^{d-1})}\leq \len{u}_{C^{d-1}(I^d)}$.

Since
	\[[\partial_{x_d}\tilde{u}](t_1,\dots,t_{d-1};x_d)=\int_{\R^{d-1}} [\partial_{x_d}u](x) \exp \frac{\iu}{2}(t_1x_1^2+\dots+t_{d-1} x_{d-1}^2)dx_1\dots dx_{d-1},\]
it is similar to get
\[\sup_{x_d\in \R} |[\partial_{x_d}\tilde{u}](t_1,\dots,t_{d-1};x_d)|\lesssim_{d-1,L}\len{\partial_{x_d} u}_{C^{d-1}(I^d)}\prod_{i=1}^{d-1}(1+|t_i|)^{-1/2}.\]
Note that $\len{\partial_{x_d} u}_{C^{d-1}(I^d)}\leq \len{u}_{C^d(I^d)}$. Then the statement for $d$ is proved.
\end{proof}

\begin{proof}[Proof of Lemma \ref{lem-loc2}]
We set $H=b_1(H_1+b_2H_2),H'=c_1(H'_1+c_2H'_2)$, then have 
	\[f_{H,H'}(k)=b_1c_1 f_{H_1+b_2H_2,H'_1+c_2H'_2}(k).\]

In Section 4, we have worked out different local normal forms of $f_{H,H'}(k)$ when $H_1,H'_1$ and $k_0$ vary. Here we should derive the local estimates according to the local normal forms. We shall keep in mind that some local estimates do saturate the upper bound (or a term of it) while some local estimates do not. Although there seem to be many different situations, we only discuss two situations here. It will be straightforward to verify the remaining cases since the arguments are essentially similar.

We first consider the simple case when $k_0$ is not a critical point of $f_{H_1,H'_1}$. We do not work on the local normal forms for this case in Section 4 because the corresponding local estimate will be derived easily in the following way. It is clear that $k_0$ remains not to be a critical point of $f_{H_1+b_2H_2,H'_1+c_2H'_2}$ when $|b_2|,|c_2|\leq \varepsilon$, here $\varepsilon$ is chosen sufficiently small. If we choose the neighborhood $U$ of $k_0$ to be sufficiently small, then the points inside $U$ have the same property as well as $k_0$. Given $\chi\in \smo_c(U,\C)$, integration by parts yields that 
	\[\left |\int_\rK u(k) \chi(k) e^{\iu f_{H,H'}(k)}dk\right|\lesssim_{U,\varepsilon,\chi} \len{u}_{C^3}(1+|b_1c_1|)^{-3} \]
for $u\in \smo(\rK,\C)$. Here the order $3$ in $C^3$ matches the exponent $-3$, which means that the integration by parts is applied for 3 times. It is obvious that
	\[(1+|b_1c_1|)^{-3}\lesssim \Omega(s^{-1}H,H')^{-1/2},\]
\[\quuad \forall s\in \rW,b_1,c_1\in \R, b_2,c_2\in (-\varepsilon,\varepsilon).\]

Next, we discuss the case of Theorem \ref{thm-n1}. We could assume the coordinate patch of $\kappa$ is of the form $(-\varepsilon,\varepsilon)\times(-\varepsilon,\varepsilon)\times U$. Then $\varepsilon$, resp. $U$ is the constant, resp. neighborhood of $k_0$ which the lemma desires. For fixed $b_2,c_2$, we let $\kappa(b_2,c_2;\cdot):U\to \R^3$ denote the map $k\to \kappa(b_2,c_2;k)$. We have 
	\[f_{H,H'}(k)=f_{H,H'}(k_0)-\frac{1}{2}\sum_{i=1}^3 s\alpha_i(H)\alpha_i(H')\cx_i^2,\]
here $\cx_i:(-\varepsilon,\varepsilon)\times(-\varepsilon,\varepsilon)\times U\to \R$ depends on $b_2,c_2$ but does not depend on $b_1,c_1$. 

Let $\chi\in \smo_c(U,\C)$. Then the integral over $\rK$ becomes an integral over $\R^3$,
	\begin{eqnarray*}
	&&\int_\rK u(k) \chi(k) e^{\iu f_{H,H'}(k)}dk\\
	&&\quad \quad \quad\quad =\pm e^{\iu f_{H,H'}(k_0)}\int_{\R^3} \tilde{u}(b_2, c_2;\cx)\exp \frac{\iu}{2} \left(\sum_{i=1}^3 -s\alpha_i(H)\alpha_i(H')\cx_i^2\right) d\cx,
	\end{eqnarray*}
here $[\kappa(b_2,c_2;\cdot)]^*\tilde{u}(b_2, c_2;\cx)d\cx=\chi(k)u(k)dk$ and the sign is determined by whether $\kappa$ preserves the orientation of $\rK$. Beware that the sign does not matter since we always consider the absolute value of the integral. 

We know that there exists $L>0$ such that $[\kappa(b_2,c_2;\cdot)]\supp \chi\subseteq [-L,L]^3$ when $|b_2|,|c_2|<\varepsilon$. Applying Lemma \ref{lem-os}, we show that the absolute value of the integral is 
	\[\lesssim_{L}\len{\tilde{u}(b_2,c_2;\cdot)}_{C^3([-L,L]^3)}\Omega(s^{-1}H,H')^{-1/2}.\]
According to the property of the norm $\len{\cdot}_{C^3(\rK,\C)}$, we have
	\[\len{\tilde{u}(b_2,c_2;\cdot)}_{C^3([-L,L]^3)}\lesssim_{\kappa,\chi} \len{u}_{C^3(\rK,\C)},\quad \forall u\in \smo(\rK,\C),b_2,c_2\in (-\varepsilon,\varepsilon).\]
Here the implicit constant depends on $\kappa$ and $\chi$ but does not depend on $u$.

Regarding other cases, we could derive the local estimate similarly. We mention a few points  in brief. In the case of Corollary \ref{cor-n4}, $\cx_3^2/2$ is replaced by $\cx_3$ in the local normal form. It will not hurt our local estimate since the integral decays faster than the bound demands.  In the case of Theorem \ref{thm-n5}, we do not know $\kappa^{-1}(b_2,c_2;0)$ when $b_2,c_2$ vary. But to apply Lemma \ref{lem-os}, it suffices to have $\kappa(b_2,c_2;\cdot)\supp \chi$ contained in a cube with a bounded size when $|b_2|,|c_2|<\varepsilon$.
\end{proof}

\subsection{Estimate for oscillatory integral \eqref{eq:ost1}}
We have established the estimate for the oscillatory integral \eqref{eq:ost2}. To deduce the estimate for the oscillatory integral \eqref{eq:ost1}, we need the following result proved by Duistermaat, which is proved for general real connected semisimple Lie groups.

\begin{theorem}[\cite{D} Theorem 1.1] 
\label{thm-dui}
There is a real analytic map $\Psi:\fs\to \rK$ such that 
	\[\HC\circ \exp\left[\Ad (\Psi(Y)^{-1})Y\right]=\pr{Y}{\fs}{\fa},\quad Y\in \fs.\] 
Furthermore, the map $\Phi_Y:\rK\to \rK, k\to k\cdot \Psi(\Ad(k^{-1})Y)$ is a diffeomorphism for each $Y\in \fs$. 
\end{theorem}

Now we deduce the estimate for the oscillatory integral \eqref{eq:ost1} by Theorem \ref{est-1}.

\begin{proposition} \label{est-2}
For any compact subset $\omega$ of $\fa$, we have
	\[\left|\int_{\rK} u(k) e^{\lambda(\HC(\exp H\cdot k))}dk\right |\lesssim_\omega \len{u}_{C^3(\rK,\C)}\sum_{s\in \rW} \Omega(s^{-1}H,-\iu \lambda^\vee)^{-1/2},\]
\[\quuad \forall  H\in \omega,\lambda\in \crFI, u\in \smo(\rK,\C).\]
\end{proposition}

\begin{proof}
Let $H'=-\iu\lambda^\vee$. 
With the diffeomorphism $\Phi_H:\rK\to \rK$ given by Theorem \ref{thm-dui}, we have
	\[\int_{\rK} u(k)e^{\lambda(\HC(\exp H\cdot k))}dk=\pm \int_\rK \Phi_H^*\Big (u(k) e^{\lambda(\HC(\exp H\cdot k))}dk\Big ),\]
here the sign is up to whether $\Phi_H$ preserves the orientation.

According to the property of $\Psi$ in Theorem \ref{thm-dui}, we have
	\[\begin{split}
		&\quad \lambda\left[\HC(\exp H\cdot \Phi_H(k))\right]=\lambda\left[\pr{\Ad(k^{-1})H}{\fs}{\fa}\right]\\
		&=\iu B(\Ad(k^{-1})H,H')=\iu f_{H,H'}(k).
	\end{split}\]
	Then the integral is equal to
	\[\pm\int_{\rK}  u(\Phi_H(k))J_H(k) e^{\iu f_{H,H'}(k)}dk,\]
here $\Phi_{H}^*(dk)=J_H(k)dk, J_H\in \smo(\rK,\C)$. Theorem \ref{est-1} tells that the absolute value of the integral is
	\[\lesssim \len{u\circ \Phi_H\cdot J_H}_{C^3(\rK,\C)}\sum_{s\in \rW}\Omega(s^{-1}H,H')^{-1/2}.\]
For fixed $H$, we have
	\[\len{u\circ \Phi_H\cdot J_H}_{C^3(\rK,\C)}\leq C_H\len{u}_{C^3(\rK,\C)},\]
here $C_H$ depends on $H$ continuously because $\Phi_H$ depends on $H$ smoothly. When $H$ belongs to a compact subset $\omega$, $C_H$ is bounded from above and the estimate is proved.
\end{proof}

At this point Theorem \ref{est-sphfun} is immediate.
\begin{proof}[Proof of Theorem \ref{est-sphfun}]
Let $u_H(k)=e^{-\rho(\HC(\exp H\cdot k))}$. We have
	\[\varphi_\lambda(\exp H)=\int_\rK u_H(k) e^{\lambda(\HC(\exp H \cdot k))}dk.\]
When $H$ belongs to the compact subset $\omega$, $\len{u_H}_{C^3(\rK,\C)}$ is bounded from above. Then Proposition \ref{est-2} implies our estimate for the spherical functions.
\end{proof}

\section{An application to eigenfunctions on locally symmetric spaces}
 
In this section, we apply Theorem \ref{est-sphfun} to the study of asymptotic behaviors of eigenfunctions on the compact quotient of $\SL(3,\R)/\SO(3)$. The main goal is Theorem \ref{est-eigen0}. Our framework involves semisimple Lie groups, globally symmetric spaces and locally symmetric spaces, so we first go through the notations and facts that will be used in the subsequent work. Some parts of the framework date back to Selberg \cite{Sel}. A comprehensive treatment could be found in \cite[Section 2 and 3]{DKV79}.


\subsection{Setting}
We start with the globally symmetric space $\sS=\SL(3,\R)/\SO(3)$. Its rank $\rk(\sS)$ is 2. Let $\invdiff(\sS)$ be the ring of $\rG$-invariant differential operators on $\sS$. It is well-known that $\invdiff(\sS)$ is generated by two differential operators, one of which is the Laplace-Beltrami operator. We consider the eigenfunctions on $\sS$ defined with respect to $\invdiff(\sS)$ and call them joint eigenfunctions. For a nonzero joint eigenfunction $\psi\in \smo(S)$, we associate to $\psi$ a homomorphism $\chi:\invdiff(\sS)\to \C$ by
	\[D\psi=\chi(D)\psi,\quad \forall D\in \invdiff(\sS).\]
We define the spectrum  $\Lambda(S)$ of $\sS$ to be the set of all such homomorphisms. On the other hand, the spherical functions when regarded as functions on $\sS$ are eigenfunctions to $\invdiff(\sS)$. For $\lambda\in \crF$, let $\chi_\lambda:\invdiff(\sS)\to\C$ be the homomorphism associated to $\varphi_\lambda$. It is known that the spherical functions are $\rW$-invariant in the spectral parameter, that is,
	 \begin{equation}\label{eq:invWeyl}
	 	\varphi_\lambda=\varphi_{s\lambda},\quad \forall s\in \rW,\lambda\in \crF.
	\end{equation}
As a result, $\chi_\lambda=\chi_{\lambda'}$ if $\lambda=s\lambda'$ for some $s\in \rW$. In fact, the converse holds. Furthermore $\Lambda(\sS)=\{\chi_\lambda|\lambda\in \crF\}$(see \cite[Proposition 3.3]{DKV79}). Then $\Lambda(\sS)$ could be identified with the quotient $\crF/\rW$ of $\crF$ under $\rW$. We say that a joint eigenfunction $\psi\in\smo(\sS)$ has the spectral parameter $\rW\cdot \lambda\in \crF/\rW$ if the homomorphism associated to $\psi$ is $\chi_\lambda$. Later, we commit a mild abuse of notation and write the spectral parameter $\lambda\in \crF$ for convenience. In other words, we treat $\Lambda(\sS)$ as $\crF$.

Next we turn to the locally symmetric space $\sX=\Gamma\backslash \sS$, here $\Gamma$ is a discrete subgroup of $\rG$ which acts freely on $\sS$. We assume that $\sX$ is compact. The members of $\invdiff(\sS)$ could be regarded as differential operators on $\sX$, to which we define the joint eigenfunctions on $\sX$ with respect. As $\sS$ is the covering space of $\sX$, we could identify continuous functions on $\sX$ with $\Gamma$-invariant continuous functions on $\sS$. So eigenfunctions on $\sX$ are equivalent to $\Gamma$-invariant eigenfunctions on $\sS$. We define the spectrum $\Lambda(\sX)$ of $\sX$ as for $\Lambda(\sS)$. Naturally $\Lambda(\sX)$ is a subset of $\Lambda(\sS)$. 

We mention some facts about $\Lambda(\sX)\subseteq \crF$. Suppose $\lambda\in \Lambda(\sX)$. \cite[Proposition 2.4]{DKV79} asserts that $\varphi_\lambda$ is positive definite. \cite[Proposition 3.4]{DKV79} tells that $\len{\lambda_R}\leq \len{\rho}$. \cite[Corollary 3.5]{DKV79} tells that there exists $s\in \rW$ such that
	\begin{equation}\label{eq:nonimg1}
		s\lambda_R=-\lambda_R,s\lambda_I=\lambda_I.
	\end{equation}
We observe that if $\lambda\notin \crFI$, then $\lambda_I$ is fixed by some nontrivial $s\in \rW$. It is evident that $\lambda_I$ is singular, that is, there exists $\alpha\in \Rr^+$ such that $\alpha(-\iu\lambda_I^\vee)=0$ (see \cite[Lemma 8.1]{DKV79}).


Our investigation employs group theoretic techniques. Generally our notations are standard, like in Helgason \cite{Helg}. For $f\in C(\rG)$, we define $\check{f},\tilde{f}\in C(\rG)$ by \[\check{f}(g)=f(g^{-1}), \tilde{f}(g)=\conj{f(g^{-1})},\quad g\in \rG.\] The convolution $*$ is defined in the common way. $C(\rG//\rK)$ is the subspace of $\rK$-bi-invariant functions in $C(\rG)$. Suppose $f\in C_c(\rG//\rK)$. For $u\in C(\sS)$, $u*f\in C(\sS)$ makes sense after identifying $u$ as a function on $\rG$. Similarly, for $u\in C(\sX)$, $u*f\in C(\sX)$ makes sense. 

 In our investigation, the Harish-Chandra transform (also called the spherical transform) will play the central role. We give it a brief account while detailed information could be found in \cite[Section 3.4]{DKV79} and \cite[Chapter 6]{GV}. For any $f\in C_c(\rG//\rK)$, its Harish-Chandra transform is the function $\crH f$ defined by
 	\[\crH f(\lambda)=\int_\rG f(g)\varphi_\lambda(g)dg, \quad \lambda\in \crF.\]
The Haar measure $dg$ on $\rG$ is suitably normalized, cf. \cite[Section 3.1]{DKV79}. Due to the $\rW$-invariance of the spherical functions \eqref{eq:invWeyl}, we have
	\[\crH f(\lambda)=\crH f(s\lambda), \quad \lambda\in \crF,s\in \rW.\]
Let $\crS(\crFI)$ be the classical Schwartz space on $\crFI$. Let $\crS(\crFI)^\rW$ be the subspace of $\rW$-invariant elements. Then for $f\in \smo_c(\rG//\rK)$, the restriction of $\crH f$ to $\crFI$ lies in $\crS(\crFI)^\rW$. The Inversion Theorem asserts that for $f\in \smo_c(\rG//\rK)$
	\begin{equation}\label{eq:idinv}
		f(g)=|\rW|^{-1}\int_{\crFI} \crH f(\nu)\varphi_{-\nu}(g)|\bc(\nu)|^{-2}d\nu, \quad g\in \rG.
	\end{equation}
$d\nu$ is the Lebesgue measure on $\crFI$ dual to the Haar measure $da$ on $\rA$. $\bc(\cdot)$ is the Harish-Chandra $\bc$-function and $|\bc(\nu)|^{-2}d\nu$ is the Plancherel measure.
	
Let $\crP(\crF)$ be the space of entire functions on $\crF$. For any $r>0$, let $\crP_r(\crF)$ be the subspace of the elements $h$ with the property: for any integer $n\geq 0$,
	\begin{equation}\label{eq:estPW}
		|h(\lambda)|\lesssim_n (1+\len{\lambda})^{-n}e^{r\len{\lambda_R}}, \quad \forall\lambda\in \crF.
	\end{equation}
Let $\crP(\crF)^\rW$, resp. $\crP_r(\crF)^\rW$ be the subspace of $\rW$-invariant elements in $\crP(\crF)$, resp. $\crP_r(\crF)$. The Paley-Wiener Theorem asserts that the Harish-Chandra transform is a bijection of $\smo_c(\rG//\rK)$ with $\crP(\crF)^\rW$. In addition, for any $r>0$, it is a bijection of the subspace of $\smo_c(\rG//\rK)$ of all functions that have supports contained in $\rG(r)$ (defined below by \eqref{eq:Gnbhd}) with the subspace $\crP_r(\crF)^\rW$.

\subsection{Main result}
We are going to study the restrictions of the joint eigenfunctions to open balls of the maximal flat subspaces in $\sX=\Gamma\backslash \sS$. For our purpose, we could assume that the maximal flat subspace is of the form
 	\[\sE(g_0,r)=\{\Gamma (g_0\exp H)\rK\in \sX| H\in \fa, \len{H}<r\},\quad g_0\in \rG,r>0.\]
Here $r$ is small enough so that the map 
\[\{H\in\fa|\len{H}<r\}\to \sX, H\to \Gamma (g_0\exp H)\rK\]
is an embedding. Besides, the measure on $\sE(g_0,r)$ is chosen to be compatible with the Haar measure on $\rA$ induced from the norm $\len{\cdot}_{\fg}$.

Throughout our discussion, we fix a positive constant $r_0$ associated to $\sX$ as follows. According to \cite[Lemma 2.7]{DKV79}, there exists an open neighborhood $U_0\subseteq\rG$ of $e$ satisfying $U_0=U_0^{-1}=KU_0K$ and $U_0\cap [\gamma]_G=\emptyset$ for $\gamma\in \Gamma, \gamma\neq e$ (here $[\gamma]_\rG$ denotes the $\rG$-conjugacy class of $\gamma$). For $r>0$, we set
	\[A(r):=\{\exp H\in \rA|H\in \fa, \len{H}\leq r\},\]
	\begin{equation}\label{eq:Gnbhd}
		\rG(r):=\rK A(r)\rK=\{\exp X\cdot k\in \rG| k\in \rK, X\in \fs, \len{X}\leq r\}.
	\end{equation}
 We take $r_0\in(0,1/2)$ such that $\rG(4r_0)$ is contained in $U_0$.

Now we are ready to state the main theorem of this section. It is a more precise version of Theorem \ref{est-eigen0}.

\begin{theorem} \label{est-eigen}
Let $\sE_1$ be the maximal flat subspace $\sE(g_0,r_1)$ in $\sX$ with $g_0\in \rG, r_1\in (0,r_0)$. If $\psi\in \smo(\sX)$ is a joint eigenfunction with  spectral parameter $\lambda\in \crF$ and $\len{\psi}_{L^2(\sX)}=1$, then we have
	\[\len{\psi}_{L^2(\sE_1)}\lesssim_{r_0} \Omega_0(-\iu \lambda_I^\vee)^{1/4}.\]
The implicit constant depends on $r_0$ but does not dependent on $g_0$.  Here $\Omega_0:\fa\to [1,\infty)$ is defined by
	\[\Omega_0(H):=\prod_{\alpha\in \Rr^+}(1+|\alpha(H)|).\]
\end{theorem}

Let us take a moment to compare our result with relevant ones.  With respect to the Laplace-Beltrami operator,  the joint eigenfunction $\psi$ with spectral parameter $\lambda\in \Lambda(\sX)$ has eigenvalue  $\len{\rho}^2+\len{\lambda_I}^2-\len{\lambda_R}^2$, see \cite[Section 3.3]{DKV79}. Since $\len{\lambda_R}\leq \len{\rho}$, we have
\[\len{\rho}^2+\len{\lambda_I}^2-\len{\lambda_R}^2\lesssim(1+\len{\lambda_I})^2.\]
Note that $\sX$ has dimension 5 and $\sE_1$ has dimension 2. Then Burq, Gérard, and Tzvetkov \cite[Theorem 3]{BGT} gives 
\[\len{\psi}_{L^2(\sE_1)}\lesssim 1+\len{\lambda_I}.\]
Marshall \cite[Theorem 1.2(c)]{M} gives a power saving \[\len{\psi}_{L^2(\sE_1)}\lesssim (1+\len{\lambda_I})^{3/4},\] 
assuming $\lambda$ lies in conic subset $[0,\infty)\cdot \omega'$ of $\crFI$, here $\omega'$ is a compact subset in $\crFI$ that is bounded away from the singular set. This result does not deal with the situation when the spectral parameter lies outside $\crF_I$. As $\rk(\sS)=2$, there could be infinitely many $\lambda\in \Lambda(\sX)$ lying outside $\crFI$, see \cite[Section 3.5]{DKV79}. We point out that our Theorem \ref{est-eigen} overcomes this drawback. Our estimate satisfies that  $\Omega_0(-\iu \lambda_I^\vee)^{1/4} \lesssim (1+\len{\lambda_I})^{3/4}$ and it is uniform in all spectral paremeters $\lambda\in \Lambda(\sX)$. 
Furthermore, if $\lambda$ is almost singular, namely, $\min_{\alpha\in \Rr^+} |\alpha(-\iu \lambda_I^\vee)|\lesssim 1$, then for such joint eigenfunctions we get a better bound 
\begin{equation}\label{eq:estsing}
\len{\psi}_{L^2(\sE_1)}\lesssim(1+\len{\lambda_I})^{1/2}.
\end{equation}
In particular, this estimate holds for joint eigenfunctions with $\lambda\notin\crFI$. Lastly, we shall clarify that our result as well as Marshall \cite{M} concerns joint eigenfunctions so it does not literally improve the result \cite{BGT} for Laplace-Beltrami eigenfunctions.


The strategy of our proof is similar to Marshall \cite[Section 3]{M}. Marshall implemented group theoretic techniques to establish his results, which was a real endeavor. In fact, group theoretic techniques have disadvantages, which are mainly caused by the higher rank, for instance, the diffculty in obtaining desired asymptotic estimate for spherical functions. The improvement in our Theorem \ref{est-eigen} is attributed to Theorem \ref{est-sphfun}, which is a better estimate for spherical functions compared with \cite[Theorem 1.3]{M} for $\rG=\SL(3,\R)$. However, the expression of our estimate is more complicated, which brings significant challenges to the method. Fortunately, the symmetries from the Weyl group $\rW$ and the root system $\Rr$ will be of great help.

\subsection{Preparations}
Now we make some preparations for the proof of Theorem \ref{est-eigen}. First, we introduce necessary notations and handy facts, which are gathered in Lemma \ref{lem-Rf} and \ref{lem-TTs}. 

\begin{lemma}\label{lem-Rf}
For any $f\in \smo_c(\rG//\rK)$, define $R(f):L^2(\sX)\to L^2(\sX), u\to u*\check{f}$.
\begin{enumerate}
\item $R(f)$  is an integral operator in $L^2(X)$ whose kernel $K_f\in \smo(\sX \times \sX)$ is given by
	\[K_f(\bar{x},\bar{x}')=\sum_{\gamma\in \Gamma} f(x^{-1}\gamma x'), \quad x,x'\in \rG,\bar{x}=\Gamma x\rK, \bar{x}'=\Gamma x'\rK.\]
\item The adjoint of $R(f)$ is $R(\tilde{f})$.
\item For $f,f'\in \smo_c(\rG//\rK)$, we have $R(f*f')=R(f)\circ R(f')$. Besides, the following identity holds
	\[
	\begin{split}
		K_{f*f'}(\bar{x},\bar{x}')&=\sum_{\gamma\in \Gamma}f*f'(x^{-1}\gamma x')\\
						     &=\int_{\sX} K_f(\bar{x},\bar{x}_1)K_{f'}(\bar{x}_1,\bar{x}')d\bar{x}_1,	
	\end{split}\]
	\[x,x'\in \rG,\bar{x}=\Gamma x\rK, \bar{x}'=\Gamma x'\rK.\]
\end{enumerate}
\end{lemma}

Lemma \ref{lem-Rf} is well-known so we skip its proof (see \cite[Proposition 2.2]{DKV79}). We clarify that the measure on $\sX$ is determined by the Haar measures on $\rG$ and $\rK$ and the counting measure on $\Gamma$.

\begin{lemma}\label{lem-TTs}
 Let $\sE_0$ be the maximal flat subspace $\sE(g_0,r_0)$ in $\sX$, with $g_0\in \rG$. For any $f\in \smo_c(\rG//\rK)$ with the support contained in $\rG(r_0)$ and $\chi\in \smo_c(E_0,\R)$, define $T:\smo(\sX)\to \smo_c(\sE_0)$,
	\[Tu(\bar{x})=\chi(\bar{x})\cdot R(f)u(\bar{x}),\quad \bar{x}\in  \sE_0.\]
$T$ could be extended to an integral operator $L^2(X)\to L^2(\sE_0)$. 

Let $T^*:L^2(\sE_0)\to L^2(\sX)$ be the adjoint of $T$. Then $TT^*:L^2(\sE_0)\to L^2(\sE_0)$ has the following expression	
		\[TT^*v(\bar{x})=\chi(\bar{x})\int_{\rA(r_0)} (f*\tilde{f})^{\vee}(a'^{-1}a)\chi(\bar{x}')v(\bar{x}')da',\]
	\[a,a'\in \rA(r_0),\bar{x}=\Gamma g_0a\rK, \bar{x}'=\Gamma g_0a'\rK.\]
\end{lemma}

\begin{proof}
We only prove the expression of $TT^*$. It suffices to show
	\[\int_\sX K_f(\bar{x},\bar{y})K_{\tilde{f}}(\bar{y},\bar{x}') d\bar{y}=(f*\tilde{f})(a^{-1}a').\]
By Lemma \ref{lem-Rf},
	\[\begin{split}K_{f*\tilde{f}}(\bar{x},\bar{x}')&=
		\int_\sX K_f(\bar{x},\bar{y})K_{\tilde{f}}(\bar{y},\bar{x}') d\bar{y}\\
		&=(f*\tilde{f})(a^{-1}a')+\sum_{\gamma\in \Gamma,\gamma\neq e} (f*\tilde{f})(a^{-1}g_0^{-1}\gamma g_0a').
	\end{split}\]
	We show that the terms in summation are all zero. For $\gamma\in \Gamma, \gamma\neq e$, we consider $g=a^{-1}g_0^{-1}\gamma g_0a'$. We shall show that $g$ does not lie in the support of $f*\tilde{f}$. Since $f$ has the support contained in $G(r_0)$, so does $\tilde{f}$. Then $f*\tilde{f}$ has the support contained in $G(2r_0)$. If $g$ lies in the support of $f*\tilde{f}$, then 
		\[g_0^{-1}\gamma g_0=aga'^{-1}\in \rA(r_0)\rG(2r_0)\rA(r_0)\subseteq \rG(4r_0).\] 
But $\rG(4r_0)$ is contained in $U_0$. Then our choice of $U_0$ leads to the contradiction. 
\end{proof}

Lemma \ref{lem-TTs} tells that if we use coordinates of $\rA$ and ignore the cutoff function $\chi$, $TT^*$ is the convolution operator on the abelian subgroup $\rA$ assoicated to the restriction of $(f*\tilde{f})^\vee$ to $\rA$. 

Next, we prove an estimate for the integral of $|\varphi_\lambda|$ over a neighborhood of $e\in \rA$, which will be used in the proof of Theorem \ref{est-eigen}.

\begin{proposition} \label{est-sphint}
	\[I(\lambda):=\int_{A(1)} |\varphi_\lambda(a)|da\lesssim \Omega_0(-\iu \lambda^\vee)^{-1/2}, \quad \lambda \in \crF_I.\]
\end{proposition}
\begin{proof} 

Denote $-\iu\lambda^\vee$ by $H'\in \fa$. We could assume that $H'$ lies in the closure of $\fa^+$ because $I$ and $\Omega_0$ are $\rW$-invariant. We could also assume $\len{H'}$ is large. 

 
Changing the variable $a=\exp H$ and applying Theorem \ref{est-sphfun}, we have
	\[\begin{split}
		I(\lambda)&=\int_{H\in \fa, \len{H}\leq 1} |\varphi_\lambda(\exp H)|dH\\
	                         &\lesssim \sum_{s\in \rW} \int_{H\in \fa, \len{H}\leq 1}  \Omega(s^{-1}H,H')^{-1/2}dH.
	\end{split}\]
We observe that the terms of integrals in the summation over $\rW$ are all equal since the integral region is $\rW$-invariant. So it suffices to estimate one term and we take the term for $s=e$. Now we divide the integral region into pieces in different Weyl chambers. We have
	\[\int_{H\in \fa, \len{H}\leq 1}  \Omega(H,H')^{-1/2}dH=\int_{H\in \fa^+,\len{H}\leq 1} \sum_{w\in \rW} \Omega(sH,H')^{-1/2}dH.\]
The summation over $\rW$ is a finite sum of positive numbers, so the summation could be estimated by looking at the largest term, that is 
 	\[\sum_{w\in \rW} \Omega(sH,H')^{-1/2}\lesssim (\min_{s\in \rW} \Omega(sH,H'))^{-1/2}.\]

We shall look into the minimum and find its explicit expression. Let $\alpha_1,\alpha_2\in \Rr^+$ be the two simple roots with $\alpha_1(H')\leq \alpha_2(H')$, $\alpha_3=\alpha_1+\alpha_2\in \Rr^+$. We have
\begin{equation}\label{eq:ordic}
0\leq\alpha_1(H')\leq \alpha_2(H')\leq \alpha_3(H')
\end{equation}

Take a basis $\{H_1,H_2\}\subseteq \fa^+$ of $\fa$ satisfying
	\[\alpha_1(H_1)=\alpha_2(H_2)=1, \quad \alpha_1(H_2)=\alpha_2(H_1)=0.\]

After calculations, we find $\len{\alpha_1}=\len{\alpha_2}=1/\sqrt{3},\len{H_1}=\len{H_2}=2$ and
	\[\{H\in \fa^+|\len{H}\leq 1\}\subseteq \{H\in \fa^+|0\leq \alpha_1(H),\alpha_2(H)\leq 1/2\}.\]
With the change of coordinates $H=x_1H_1+x_2H_2, x_1,x_2\in (0,1/2)$, we have
	\[I(\lambda)\lesssim \iint_{(0,1/2)^2} (\min_{s\in \rW} \Omega(s(x_1H_1+x_2H_2),H'))^{-1/2}dx_1dx_2.\]
For $H=x_1H_1+x_2H_2\in \fa^+$, the set of values $\{|\alpha(sH)||\alpha\in \Rr^+\}$ does not change when $s$ varies and is always $\{x_1,x_2,x_1+x_2\}$. Due to \eqref{eq:ordic}, $\Omega(sH,H')$ $(s\in \rW)$ must be minimal if
\begin{equation}\label{eq:orddc}
    |\alpha_1(sH)|\geq |\alpha_2(sH)|\geq |\alpha_3(sH)|.
\end{equation} 
And we can verify that given $H\in \fa^+$ satisfying \eqref{eq:ordic}, there exists $s\in \rW$ such that the inequality \eqref{eq:orddc} holds. Let $c_i=\alpha_i(H')$, for $i=1,2,3$. Therefore, we have
\[\begin{split}
&\quad\min_{s\in \rW} \Omega(s(x_1H_1+x_2H_2),H')\\
&=\begin{cases}
	(1+c_1(x_1+x_2))^{-1/2}(1+c_2x_2)^{-1/2}(1+c_3x_1)^{-1/2} & 0<x_1\leq x_2<\frac{1}{2}\\
	(1+c_1(x_1+x_2))^{-1/2}(1+c_2x_1)^{-1/2}(1+c_3x_2)^{-1/2} &0< x_2<x_1<\frac{1}{2}
\end{cases}
\end{split}
\]

As the integral region $(0,1/2)^2$ is symmetric in $x_1,x_2$, the integral over $0<x_1\leq x_2<1/2$ is equal to the integral over $0<x_2<x_1<1/2$. So we only need to deal with the part $0<x_1\leq x_2<1/2$,
	\begin{equation}\label{eq:intdouble}
		\int_0^{1/2}dx_1\int_{x_1}^{1/2} (1+c_1(x_1+x_2))^{-1/2}(1+c_2x_2)^{-1/2}(1+c_3x_1)^{-1/2}dx_2.
	\end{equation}
If $c_1\leq 2$, then the integral \eqref{eq:intdouble} is less then
	\[\int_0^{1/2}dx_1\int_{x_1}^{1/2} (c_2x_2)^{-1/2}(c_3x_1)^{-1/2}dx_2\lesssim (c_2c_3)^{-1/2} \]
If $c_1>2$, then the integral \eqref{eq:intdouble} is less than
	\[\begin{split} 
		&\quad\int_0^{1/2}dx_1\int_{x_1}^{1/2} (c_1x_2)^{-1/2}(c_2x_2)^{-1/2}(c_3x_1)^{-1/2}dx_2\\
		&=(c_1c_2c_3)^{-1/2}\int_0^{1/2}x_1^{-1/2}(\ln(1/2)-\ln x_1)dx_1\\
		&\lesssim (c_1c_2c_3)^{-1/2}.	
	\end{split}\]
Combining two situations, the proof is done.
\end{proof}

\subsection{Proof} 

\begin{proof}[Proof of Theorem \ref{est-eigen}] 

We begin the proof by showing the existence of  $h\in \crP_{r_0}(\crF)^\rW$ that satisfies 
	\begin{equation}\label{eq:condition1}
		h(\nu)\geq 0, \quad \forall\nu\in \crF_I
	\end{equation}
and
	\begin{equation}\label{eq:condition2}
		|h(\nu)|\geq 2,\quad \forall \nu\in \crF, \len{\nu}\leq\len{\rho}.
	\end{equation}

The existence of such $h$ is the starting point of our argument. We use the Harish-Chandra transform to construct such $h$.

As $\varphi_\nu(e)=1$ for all $\nu\in \crF$, there exists $r_2<r_0/2$ such that for all $g\in \rG(r_2)$ and $\nu\in \crF$ with $\len{\nu}\leq \len{\rho}$, we have
 	\begin{equation}\label{eq:Gnbhd3}
 		|\varphi_\nu(g)-1|\leq \frac{1}{3}.
 	\end{equation}	
Now we take a real non-negative function $f\in \smo_c(\rG//\rK)$ that has the support contained in $G(r_2)$ and satisfies $\int_\rG f(g)dg=3$. Regarding its Harish-Chandra transform, we have that for $\nu\in \crF$,
	\[\crH f(\nu)=\int_{\rG(r_2)} f(g)[1+(\varphi_\nu(g)-1)]dg=3+\int_{\rG(r_2)} f(g)(\varphi_\nu(g)-1)dg.\]	
For $\nu\in \crF$ with $\len{\nu}\leq \len{\rho}$, \eqref{eq:Gnbhd3} gives
	\[\Big |\int_{\rG(r_2)} f(g)(\varphi_\nu(g)-1)dg \Big|\leq \int_{\rG(r_2)} f(g)\cdot \frac{1}{3}dg=1.\]
Therefore, for $\nu\in \crF$ with $\len{\nu}\leq \len{\rho}$, we have $|\crH f(\nu)|\geq 2$.  

Define $h$ to be $\crH (f*\tilde{f})$. We verify that $h$ satisfies conditions \eqref{eq:condition1} and \eqref{eq:condition2}. First, $f*\tilde{f}$ has the support contained in $\rG(2r_2)$.  As $r_2<r_0/2$, the Paley-Wiener theorem asserts that $h\in \crP_{r_0}(\crF)^\rW$. As $\crH \tilde{f}(\nu)=\conj{[\crH f(-\conj{\nu})]}$ for $\nu\in\crF$,
	\[\crH (f*\tilde{f})(\nu)=\crH f(\nu)\cdot \crH \tilde{f}(\nu)=\crH f(\nu) \cdot \conj{[\crH f(-\conj{\nu})]}, \quad \nu\in \crF.\]
For $\nu\in \crFI$, we have $-\conj{\nu}=\nu$ and $\crH (f*\tilde{f})(\nu)=|\crH f(\nu)|^2\geq 0$. For $\nu\in \crF$ with $\len{\nu}\leq \len{\rho}$, we have $\len{-\conj{\nu}}=\len{\nu}\leq \len{\rho}$ and $|\crH (f*\tilde{f})(\nu)|\geq 4$.

According to the property \eqref{eq:estPW} of  $\crP_{r_0}(\crF)^\rW$, there exists $L>0$ such that
	\[|h(\nu)|\leq \frac{1}{100},\quad \forall \nu\in \crF, \len{\nu_R}\leq \len{\rho},\len{\nu_I}\geq L.\]

We emphasize that $h$ is fixed when we consider all the joint eigenfunctions. We also remark that our construction of $h$ depends on $r_0$, so in the end our estimate depends on $r_0$. 

For $\lambda\in \Lambda(\sX)$ with $\len{\lambda_I}\geq 2L$, we put
	\begin{equation}\label{eq:sumWeyl}
		h_\lambda(\nu)=\sum_{s\in \rW} h(s\nu-\lambda_I), \quad \nu\in \crF.
	\end{equation}
We remark that we could skip finitely many joint eigenfunctions because our statement is asymptotic. It is clear that there are only finitely many $\lambda\in \Lambda(\sX)$ with $\len{\lambda_I}<2L$.

It is not hard to verify that $h_\lambda\in \crP_{r_0}(\crF)^\rW$. Then the Paley-Wiener Theorem asserts that there exists $f_\lambda\in\smo_c(\rG//\rK)$ satisfying $\crH f_\lambda=h_\lambda$ and its support is contained in $G(r_0)$.

In the upcoming argument, we want the value of $h_\lambda$ at $\lambda$ to stay away from 0. We prove the following claim: for $\lambda\in \Lambda(\sX)$ with $\len{\lambda_I}\geq 2L$,
	\begin{equation}
		|h_\lambda(\lambda)|\geq 1.
	\end{equation}	
We consider two situations $\lambda\in \crFI$ and $\lambda\notin\crFI$. 

For $\lambda\in \crFI$, we have $h_\lambda(\lambda)\geq h(0)\geq 2$ because $h(\nu)\geq 0$ for $\nu\in\crFI$.

For $\lambda\notin\crFI$, we recall that $\lambda_I$ satisfies the property \eqref{eq:nonimg1}, that is, there exists $s\in \rW$ such that $s\lambda_I=\lambda_I$. We consider the terms $h(s\lambda_R+s\lambda_I-\lambda_I)$ in the summation \eqref{eq:sumWeyl}. If $s\lambda_I=\lambda_I$, then $h(s\lambda_R)=h(\lambda_R)$ since $h$ is $\rW$-invariant. If $s\lambda_I\neq \lambda_I$, then $\len{s\lambda_I-\lambda_I}\geq \len{\lambda_I} \geq L$ and $|h(s\lambda_R+s\lambda_I-\lambda_I)|\leq 1/100$. Here we use the geometry of the root system and the property that $\lambda_I$ is singular. Note that $\rW$ has 6 elements. We conclude that $|h_\lambda(\lambda)|\geq 2|h(\lambda_R)|-4/100\geq1$.

So far we have constructed a family of test functions $f_\lambda\in \smo_c(\rG//\rK)$ of which their Harish-Chandra transform $h_\lambda$ satisfy certain properties. We are going to use them to construct the spectral projectors. Then we will deduce our theorem after deriving the estimates for some operator norms.

Let $\sE_0=\sE(g_0,r_0)$. Take a cutoff function $\chi\in \smo_c(\sE_0)$ satisfying $0\leq \chi(\bar{x})\leq 1$, for $\bar{x}\in \sE_0$ and $\chi(\bar{x})=1$ for $\bar{x}\in \sE_1$. We define $R(f_\lambda)$ and $T_\lambda$ as in Lemma \ref{lem-Rf} and Lemma \ref{lem-TTs}.
	As the eigenfunction $\psi$ has the spectral parameter $\lambda\in \crF$, the uniqueness of spherical functions yields
		\[R(f_\lambda)\psi(\bar{x})=\crH f_\lambda(\lambda)\psi(\bar{x})=h_\lambda(\lambda) \psi(\bar{x}),\quad \bar{x}\in \sX.\]
Therefore
		\[T_\lambda \psi(\bar{x}')=h_\lambda(\lambda)\chi(\bar{x}')\psi(\bar{x}'),\quad \bar{x}'\in \sE_0.\]

	As $|h_\lambda(\lambda)|\geq 1$ and $\chi$ takes 1 on $\sE_1$, we have 
		\[\len{\psi}_{L^2(\sE_1)}\leq \len{T_\lambda \psi}_{L^2(\sE_0)}.\]
	 Since $\len{\psi}_{L^2(\sX)}=1$, $\len{T_\lambda \psi}_{L^2(\sE_0)}\leq \len{T_\lambda}_{L^2(\sX)\to L^2(\sE_0)}$. Meanwhile, \[\len{T_\lambda}_{L^2(\sX)\to L^2(\sE_0)}=\len{T_\lambda T_\lambda^*}^{1/2}_{L^2(\sE_0)\to L^2(\sE_0)}.\]
	 So the theorem will be proved if we establish the desired estimate for $T_\lambda T_\lambda^*$. By Lemma \ref{lem-TTs} and Young's inequality, we have 
		\[\len{T_\lambda T_\lambda^*}_{L^2(\sE_0)\to L^2(\sE_0)}\leq \len{f_\lambda*\tilde{f}_\lambda}_{L^1(\rA)}.\] 
Hence our task becomes to estimate $\len{f_\lambda*\tilde{f}_\lambda}_{L^1(\rA)}$. $\tilde{f}_\lambda=f_\lambda$ since $\crH f_\lambda=h_\lambda$ is real on $\crFI$. Since $\supp f_\lambda\subseteq G(r_0)$, $\supp f_\lambda*f_\lambda\subseteq G(2r_0)$. As a result,
		\[\len{f_\lambda*f_\lambda}_{L^1(\rA)}= \len{f_\lambda*f_\lambda}_{L^1(\rA(2r_0))}= \len{f_\lambda*f_\lambda}_{L^1(\rA(1))}.\]
	As $\crH(f_\lambda*f_\lambda)(\nu)=(\crH f_\lambda(\nu))^2=h_\lambda(\nu)^2$ for $\nu\in \crF$, we have for $a\in \rA(1)$
		\[\begin{split}
			|f_\lambda*f_\lambda(a)|&=|\rW|^{-1}\Big |\int_{\crFI} h_\lambda(\nu)^2\varphi_{-\nu}(a)|\bc(\nu)|^{-2}d\nu \Big|\\
							      &\lesssim \int_{\crFI} h_\lambda(\nu)^2|\varphi_{-\nu}(a)||\bc(\nu)|^{-2}d\nu.
		\end{split}\]
	By Cauchy-Schwarz inequality, for $\nu\in \crFI$,
		\[h_\lambda(\nu)^2=\Big(\sum_{s\in\rW} h(s\nu-\lambda_I)\Big)^2\lesssim \sum_{s\in \rW} h(s\nu-\lambda_I)^2.\]
	Then
		\[\len{f_\lambda*f_\lambda}_{L^1(\rA)}\lesssim \sum_{s\in \rW} \int_{\rA(1)}da \int_{\crFI} h(s\nu-\lambda_I)^2 |\varphi_{-\nu}(a)||\bc(\nu)|^{-2}d\nu.\]
	The summation over $\rW$ could be reduced to one term because $\varphi_\cdot(a)$ and $|\bc(\cdot)|^{-2}$ are $\rW$-invariant, which gives
		\[\len{f_\lambda*f_\lambda}_{L^1(\rA)}\lesssim \int_{\rA(1)}da \int_{\crFI} h(\nu-\lambda_I)^2 |\varphi_{-\nu}(a)||\bc(\nu)|^{-2}d\nu.\]
	Changing the order of the integration and applying Proposition \ref{est-sphint}, it turns out
		\[\len{f_\lambda*f_\lambda}_{L^1(\rA)}\lesssim \int_{\crFI} h(\nu-\lambda_I)^2\Omega_0(\iu\nu^\vee)^{-1/2}|\bc(\nu)|^{-2}d\nu.\]
	As $|\bc(\nu)|^{-2}\lesssim \Omega_0(-\iu\nu^\vee)=\Omega_0(\iu\nu^\vee)$ cf. \cite[Section 3.8]{DKV79},
		\[\len{f_\lambda*f_\lambda}_{L^1(\rA)}\lesssim \int_{\crFI} h(\nu-\lambda_I)^2\Omega_0(-\iu\nu^\vee)^{1/2}d\nu.\]
Because $\Omega_0(-\iu\nu^\vee)\leq \Omega_0(-\iu\nu^\vee+\iu\lambda_I^\vee)\Omega_0(-\iu\lambda_I^\vee)$ for $\nu\in \crF_I$ ,  we have
		\[\len{f_\lambda*f_\lambda}_{L^1(\rA)}\lesssim \Omega_0(-\iu\lambda_I^\vee)^{1/2} \int_{\crFI} h(\nu-\lambda_I)^2\Omega_0(-\iu\nu^\vee+\iu \lambda_I^\vee)^{1/2}d\nu.\]
	Changing the coordinates, we get
		\[\int_{\crFI} h(\nu-\lambda_I)^2\Omega_0(-\iu(\nu-\lambda_I)^\vee)^{1/2}d\nu=\int_{\crFI} h(\nu)^2\Omega_0(-\iu\nu^\vee)^{1/2}d\nu.\]
	This integral is convergent because $h^2$ restricted to $\crFI$ decays rapidly and $\Omega_0(-\iu \nu^\vee)\lesssim (1+\len{\nu})^3$.
	We get 
		\[\len{f_\lambda*f_\lambda}_{L^1(\rA)}\lesssim \Omega_0(-\iu\lambda_I^\vee)^{1/2}.\]
	In the end, we have
		\[\len{\psi}_{L^2(\sE_1)}\leq\len{f_\lambda*f_\lambda}_{L^1(\rA)}^{1/2}\lesssim  \Omega_0(-\iu\lambda_I^\vee)^{1/4}.\]

	
\end{proof}


%
%



\end{document}